\documentclass[10pt]{article}
\usepackage{amsmath}
\usepackage{amssymb}
\usepackage{amsthm}
\usepackage{graphicx}
\usepackage{verbatim}
\usepackage{enumerate}
\usepackage{color}
\usepackage{xcolor}
\definecolor{cgreen}{rgb}{0.0,0.42.0.24}
\definecolor{cpurple}{rgb}{0.78,0,0.82}
\usepackage{mathrsfs}
\usepackage[colorlinks=true,linkcolor=blue,citecolor=blue,pdfpagelabels=false]{hyperref}
\usepackage{longtable}
\numberwithin{equation}{section}
\usepackage[english]{babel}
\usepackage{geometry,amsmath,amsthm,amssymb,latexsym,graphicx}
\usepackage{tikz}

\usepackage{algpseudocode}
\usepackage[algoruled,boxed,lined]{algorithm2e}

\usepackage{fancyvrb} 
\fvset{formatcom=\color{red}} 
\fvset{listparameters={\setlength{\topsep}{0pt}\setlength{\partopsep}{0pt}}} 

\usepackage{environ}

\newcommand{\eqb}{\begin{equation}}
\newcommand{\eqe}{\end{equation}}

\theoremstyle{plain}
\newtheorem{theorem}{Theorem}[section]

\newtheorem{lemma}[theorem]{Lemma}
\newtheorem{proposition}[theorem]{Proposition}
\newtheorem{corollary}[theorem]{Corollary}
\newtheorem{definition}[theorem]{Definition}

\newtheorem{example}{Example}

\newtheorem{problem}{Problem}
\theoremstyle{remark}

\newtheorem{remark}{Remark}

\vfill

\allowdisplaybreaks

\font\elevenss=cmss11

\font\eightss=cmss8

\font\sixss=cmss8 at 6pt

\newfam\ssfam
\textfont\ssfam=\elevenss \scriptfont\ssfam=\eightss
 \scriptscriptfont\ssfam=\sixss
\def\ss{\fam\ssfam \elevenss}%

\def\N{\mathbb{N}}
\def\Z{\mathbb{Z}}

\def\R{\mathbb{R}}
\def\RP{\mathbb{RP}}
\def\C{\mathbb{C}}
\def\CP{\mathbb{CP}}
\def\logspace{{\mathcal L}}
\def\CC{{\cal C}}
\def\AA{{\cal A}}

\def\ee{\varepsilon}
\def\dd{\delta}

\def\disp{\displaystyle}
\def\one{{\bf 1}}
\def\zero{\bf{0}}
\def\|{{\, | \, }}

\def\xx{{\bf x}}
\def\yy{{\bf y}}
\def\rr{{\bf r}}
\def\zz{{\bf z}}

\def\ww{{\bf w}}
\def\vv{{\bf v}}

\def\mm{{\bf m}}

\def\fC{\mathfrak{C}}
\def\fF{\mathfrak{F}}
\def\fH{\mathfrak{H}}
\def\M{{\cal M}}

\def\D{{\cal D}}

\def\nbd{{\mathcal N}}
\def\grad{{\nabla}}

\def\crit{\mbox{\ss crit}}
\def\zero{{\bf 0}}
\def\sing{{\mathcal V}}
\def\TT{{\bf T}}
\def\rhat{{\hat{\rr}}}

\def\Res{{\rm Res}\,}

\def\amoeba{\mbox{\ss amoeba}}

\def\I{{\mathcal I}}

\def\tsing{\widetilde{\sing}}
\def\T{\mathbb{T}}
\def\V{\mathbb{V}}
\def\newpol{\mathbf{P}}
\def\expdecay{\simeq}

\def\str{{\Sigma}}

\def\J{\mathbf{J}}

\def\tsing{\tilde{\sing}}

\def\rk{\mathrm{rank}}

\def\flow{\Psi}
\def\bzero{{\mathbf 0}}
\def\spai{{\rm SPAI}}
\def\hspai{{\rm H{\textnormal -}SPAI}}

\def\leray{\mathtt{o}}
\def\disk{\bullet}
\def\chain{{\gamma}}
\def\Comp{\mathbb{C}}
\def\Real{\mathbb{R}}

\def\hr{{h_\rr}}
\def\htt{\tilde{h}}

\def\compact{{\mathcal K}}
\def\Sb{{\overline{S}}}

\begin{document}

\begin{titlepage}

\begin{center}
{\LARGE \bf Stationary points at infinity for analytic combinatorics}
\end{center}

{%\large

\noindent{\bf Abstract:}
On complex algebraic varieties, height functions arising in 
combinatorial applications fail to be proper. This complicates
the description and computation via Morse theory of key topological 
invariants.  Here we establish checkable conditions under which
the behavior at infinity may be ignored, and the usual theorems of
classical and stratified Morse theory may be applied.
This allows for simplified arguments in the field of
analytic combinatorics in several variables, and forms the basis
for new methods applying to problems beyond the reach of previous
techniques.
\vfill
\vfill

{\sc \small
Yuliy Baryshnikov,
University of Illinois,
Department of Mathematics,
273 Altgeld Hall 1409 W. Green Street (MC-382), 
Urbana, IL 61801, {\tt ymb@illinois.edu},
partially supported by NSF grant DMS-1622370. 

Stephen Melczer,
Department of Combinatorics \& Optimization,
University of Waterloo, 
200 University Avenue West,
Waterloo, ON N2L 3G1, Canada, {\tt smelczer@uwaterloo.ca},
partially supported by an NSERC postdoctoral fellowship. 

Robin Pemantle,
University of Pennsylvania,
Department of Mathematics,
209 South 33rd Street,
Philadelphia, PA 19104, {\tt pemantle@math.upenn.edu},
partially supported by NSF grant DMS-1612674.  
}


}
\end{titlepage}

\section{Introduction} \label{sec:intro}  

\subsection{Motivation from Analytic Combinatorics} \label{ss:motivation}
 
Analytic combinatorics in several variables (ACSV) studies
coefficients of multivariate generating functions via analytic
methods; see, for example,~\cite{PW-book}. 
The most developed part of the theory is the asymptotic determination
of coefficients of multivariate series $F(\zz) = \sum_\rr a_\rr \zz^\rr$
where the coefficients $a_\rr$ are defined by
a multivariate Cauchy integral
\begin{equation} \label{eq:cauchy}
a_\rr = \frac{1}{(2\pi i)^d}\int_T \zz^{-\rr} F(\zz) \frac{d\zz}{\zz} \, ,
\end{equation}
for an appropriate torus of integration $T\subset\C^d$.  In many applications, 
$F(\zz) = P(\zz)/Q(\zz)$ is rational function with a power series 
expansion whose coefficients are indexed by $\rr$, an integer vector.  
More generally, one often looks for formulas valid as $\rr$ varies
over some cone.  %%% over some region $K$ in projective space $\RP^{d-1}$.  

Let $|\rr| := \sum_{j=1}^d |r_j|$ denote the $\ell_1$ norm and
let $\rhat$ denote the scaled vector $\rr/|\rr|$.  As a slight abuse of
notation, we will sometimes consider $\rhat$ to be an element of 
$\RP^{d-1}$ rather than the $\ell^1$ unit ball, when the implicit
identification of $\pm\rhat$ leads to no ambiguity.
%%%RP210218: see slight changes above in discussion of cone, \rhat, etc.

Given any $\rr \in \R^d$ we define the \emph{phase function}, 
or \emph{height function}, depending only on $\rhat$, by
\begin{equation}
h(\zz) = h_\rr (\zz) = h_{\rhat} (\zz) =  - \Re(\rhat \cdot \log \zz) = 
   - \sum_{j=1}^d \hat{r}_j \log|z_j| \, , 
\label{eq:phase}
\end{equation}
where the logarithm is taken coordinate-wise and 
$\Re(z)$ denotes the real part of complex $z$. 
The height function is useful because it captures the 
behaviour of the term $|\zz^{-\rr}| = \exp ( |\rr| \hr (\zz))$
in the Cauchy integral that grows with~$\rr$.  
Note that even though the ratio $\rhat  = \rr/ |\rr|$ is
always rational, a sequence of such direction vectors may converge
to any real $\rhat$ and our results often hold uniformly for
$\rhat$ in cones of $\RP^{d-1}$. We will be clear when results
do not require $\rhat$ to be rational.

Typically, the Cauchy integral~\eqref{eq:cauchy} is evaluated by applying the stationary 
phase (saddle point) method after a series of deformations of the chain 
of integration.  
To elaborate, we let $\sing := \{ \zz : Q(\zz) = 0 \}$ denote the pole
variety of rational $F(\zz)=P(\zz)/Q(\zz)$, with $P$ and $Q$
coprime, let $\C_*$ denote the nonzero complex numbers, and let 
$\M := \C_*^d \setminus \sing$ denote the domain of holomorphy of
the Cauchy integrand in~\eqref{eq:cauchy}.  The Cauchy integral depends 
only on the homology class\footnote{Throughout, we assume integer 
coefficients for all homology groups.} $[T]$ of $T$ in $H_d (\M)$. 
Stratified Morse theory for complements of closed, Whitney stratified 
spaces suggests that any cycle can be deformed downward (in the sense of 
decreasing $\hr$) until it reaches a topological obstacle at some 
(stratified) critical point of $\sing$.  A topological obstacle 
generally implies that a stationary phase contour has been reached.  
This leads to the following outline for the asymptotic evaluation of 
coefficients $a_\rr$.

\begin{enumerate}[(I)]
\item Find a basis of the singular homology group with integer 
  coefficients $H_d (\M)$, consisting of attachment 
chains 
$\{ \sigma_k \}$ localized near the near critical points $p_k$ of $\hr$ on the
Whitney stratified space $\sing$, in the sense that each chain $\sigma_k$ intersected with
the set of points $\{\zz : \hr(\zz) > \hr(p_k)-\epsilon \}$ 
is contained in a ball around $p_k$ of radius shrinking with $\epsilon$.
\item Compute the coefficients of $[T]$ in this basis; that is, write
\begin{equation} \label{eq:basis} 
[T] = \sum_k n_k \sigma_k
\end{equation}
for some integers $\{ n_k \}$.
\item Asymptotically compute the Cauchy integral over each 
chain of integration $\sigma_k$.
\end{enumerate}

This paper concerns pre-conditions for the validity of this program.
In reverse order, we discuss parts of the program carried out elsewhere.

Part III of the program varies in difficulty depending on the
nature of the critical point.  When the point is in a stratum
of positive dimension
where $\sing$ has at worst normal intersections, it is at worst a
multivariate residue together with a saddle point integral, the
integral being somewhat more tricky if the saddle point is not
quadratically nondegenerate.  This follows from~\cite{varchenko76}
or~\cite{vassiliev77} and is stated explicitly in~\cite{PW2} 
and~\cite[Chapter~10]{PW-book}.  The more difficult case is when 
the critical point $p$ is an isolated singularity
of the stratified space $\sing$.
Homogenizing at $p$, one reduces to the problem of computing the
inverse Fourier transform of a homogeneous hyperbolic function,
some instructions for which can be found in~\cite{ABG}.  This is 
carried out in~\cite{BP-cones} for two classes of quadratic singularities, 
and in~\cite{BMP-lacuna} for some singularities of lacuna type. 
For non-isolated singularities, or isolated singularities of higher degree, 
the technology is still somewhat {\em ad hoc}.

Step II is a topological computation.  In a slightly different context, 
Malgrange~\cite{malgrange80} noted the lack of techniques for approaching 
a similar problem.  Effective algorithms exist only
in special cases, such as the bivariate case~\cite{DVvdHP2012}. 
A new computational approach, relying on the results of the 
present article together with techniques from computer algebra,
is discussed in~\cite{BMP-lacuna} and below. 

Step~I may fail entirely. It is not always true that such a 
basis exists (see Examples~\ref{eg:actual} and~\ref{eg:stat} below),
due to the existence of a topological obstruction at infinity\footnote{Here
and throughout, ``infinity'' refers not only to projective points
but to points where at least one coordinate vanishes; these are the
cases in which the height function may not be well defined. Affine
points with coordinates equal to zero may arise as critical points
for Laurent series.}.  
%%%RP210218: Check footnote, inserted so as to include coordinate planes
%%%in the notion of "at infinity".
The focus of the present paper is to find checkable conditions 
under which the class $[T]$ is indeed representable in the 
form~\eqref{eq:basis}.  This has been a sticking point
up to now in the development of ACSV methods.

\subsection{Previous work} \label{ss:previous}

Although the methods of ACSV parallel well-established
mathematical techniques, the underlying combinatorics 
often results in constraints which are natural in our
context but not covered by existing theory.
In this section we discuss related previous work
and why the results we need do not follow from it.  

To begin, there are several reasons why stratified Morse theory
does not immediately imply the existence of the type of basis
appearing in Step~I.  If $\hr$ were a Morse function 
(in the stratified sense) then Theorems~A and~B of~\cite{GM} 
would in fact imply that such a basis exists with some number $m_k$
of generators associated with each critical point $p_k$.  These
are given by the rank of a relative homology group of the normal link at $p_k$
of the stratum containing $p_k$ (in the dimension equal to the codimension of the stratum).
If $\hr$ fails to be Morse by behaving degenerately at $p_k$
such a basis still exists, though it might take a messy perturbation 
argument to compute the rank at $p_k$ and give cycle representatives
for a basis.  

A more serious problem for us is that $\hr$ is not a proper function 
on $\sing$.  This means that gradient flows of the Cauchy domain of 
integration may reach infinity or the coordinate planes at some finite
%%%RP210218: added coordinate planes here
height $c$, and hence that contours may not be deformable to
levels below $c$ because they get sucked out to infinity first.
As shown in the examples in Section~\ref{sec:examples}, this
can indeed happen.  

A somewhat generic cure for this is to compactify.  In the appendix,
we outline how to embed $\M$
in a compactification $X$
to which the phase function $\hr$ and its gradient extend continuously.  
Applying the results of stratified Morse theory to $X$ then decomposes
the topology of $X$ into attachments at critical points of $X$.  Generically,
there will be finitely many critical points of $X$, all lying in $\M$. 
When this occurs, $\M$ is said to have no stationary points at infinity
and the decomposition in Step~I follows.
Two weaknesses of this approach
are the difficulty in computing $X$ (it relies on an unspecified 
resolution of singularities) and the fact that $X$ depends on $\rr$
not continuously but rather through the arithmetic properties of
the rational vector $\rhat$, thus failing to deliver asymptotics
uniform in a region.  

\subsubsection{Related notions of singularities at infinity}

We now review three streams of prior work where, under some hypotheses
of avoidance of critical points at infinity, the topology of a space
is shown to decompose similarly to the desired decomposition in 
Step~I.

One setting where such problems have been investigated concerns the 
Fourier transform of $\exp (f)$ where $f: \R^d \to \R$ is a real 
polynomial.  When $f$ is homogeneous, resolution of singularities puts 
the phase into a monomial form, after which leading term asymptotics
can be read from its Newton diagram~\cite{vassiliev77,varchenko76}.
For general polynomials, critical points of $f$ occur at places
other than the origin and these local integrals must be pieced
together according to some global topological computation\footnote{Pham
attributes this idea to Malgrange, ``The reduction of a {\em global} 
Fourier-like integral to a sum of Laplace-like integrals is the
topic of Malgrange's recent paper, motivated by an idea of
Balian-Parisi-Voros.''~\cite{Pham1983}.}.
A key difference from our case is that this is an
integral over all of $\R^d$ (there is no polar set).  
Morse theory enters the picture via the filtration 
$\{ \Im (f) \leq c \}_{-\infty < c < \infty}$, which dictates
how the contour of integration may be deformed, where
$\Im(f)$ denotes the imaginary part of $f$.

Many similarities between this case and ours are evident.  
The height function $\Im \{ f \}$ plays a similar role to
our height function $\hr$.  
One may apply methods of saddle point integration at the 
critical points of $f$, as done by Fedoryuk~\cite{fedoryuk-book} 
and Malgrange~\cite{malgrange74,malgrange80}.  Fedoryuk computes
in relative homology, which is good enough for the estimation 
of integrals. 
Pham~\cite{Pham1983} uses an absolute homology theory over a family 
of supports, enabling more precise asymptotic results.  
Pham's crucial hypothesis H1 is that there are no
bifurcation points of the second type or ``critical points at infinity''.
The conclusion is the existence of a basis for the homology of
$\C^d$ with downward supports consisting of so-called {\em Lefschetz
thimbles}, along with a dual basis (in the sense of intersections)
allowing one to compute the coefficients of an arbitrary cycle
in this basis.  Unfortunately, we can see no direct connection
that would reduce our computation to the type analyzed by
Pham and others before him.  Even if we could, the issue of spurious
critical points at infinity would still need to be addressed.
As pointed out in~\cite[page~330]{Pham1983}, there is no simple 
way of telling which of these is relevant.

A second stream of work concerns the topology of complex
polynomial hypersurfaces.  Here there are no integrals, 
hence no phase functions {\em per se}, although motivation
from~\cite{malgrange80} and~\cite{Pham1983} is cited in the
introduction of~\cite{Broughton1988}.  In this paper, 
Broughton computes the homotopy type of a generic
level set $f^{-1} (c)$, showing it to be a bouquet of
$n$ spheres, where $n$ is obtained by summing the Milnor numbers
at the critical points of $f$ other than those with critical
value $c$. Examples show that a hypothesis is necessary
to rule out ``critical points at infinity''; see also~\cite{Parusinski1995}.
Both of these works refer to an assumption of only isolated
critical points at infinity or none at all, but do not supply
a specific definition of critical points at infinity.  Such a
definition is supplied in~\cite{siersma-tibar1995}.  They 
compactify $f$ by taking the closure of its graph in projective space 
and taking a Whitney stratification of the resulting relation.
Depending on whether there are no critical points at infinity,
void (hence ignorable) critical points at infinity, isolated 
non-void critical points at infinity, or non-isolated points,
various conclusions can be drawn about the topologies of the
fibers $f^{-1} (c)$.  Again, no direct relation allows us to
derive from this the decomposition in Steps~I and~II, and 
even then, the issue of spurious critical points would remain.

A third stream of work concerning critical values at infinity
comes closest to our aims here.  The focus of this stream of
work is, given a smooth map $F: \M \to \N$ on some sort of space,
to find a set $B \subseteq \N$ that is not too big, such that
outside of $F^{-1} [B]$, the mapping is a locally trivial fibration.
For proper maps one has the Thom isotopy lemma, which states that
$B$ can be taken as the set $K_0$ of critical values, namely
the set $F(x)$ where $F$ is not a submersion at $x$, failing
to map the tangent space surjectively.  To extend the isotopy
lemma to nonproper maps, one needs to add to $K_0$ an apprpriately
defined set $K_\infty$ of critical values at infinity.  When $\M$
is smooth the so-called Palais-Smale Condition, and
Rabier's general notion of {\bf asymptotic critical 
value}~\cite[Section~6]{rabier1997}, yields such an isotopy result,
valid in a quite general infinite dimensional setting.  Further
work has shown that under reasonable hypotheses the set of 
critical values $K_\infty$ is not too big, for example it has 
measure zero~\cite{kurdyka-orro-simon2000}. 	

Our work requires a result along similar lines, but for stratified spaces.
More specifically, in the proof of Theorem~\ref{th:main} below,
we use an isotopy result away from a small computable set of
stratified critical values for the height function
to describe $H_d (\M)$ in terms of stationary
phase contours.  A similar project is undertaken in the contemporaneous
work~\cite{dinh-jelonek2021}.  There, Dinh and Jelonek prove a version
of the (stratified) Thom isotopy lemma away from a computable 
and nowhere dense set $K_0 \cup K_\infty$ of affine and at-infinity 
critical values. In comparison to our work, Dinh and Jelonek
work in a more general setting but thus use more complicated
constructions leading to a less practical algorithm for detecting critical 
points at infinity. We thus do not use their stratified non-proper isotopy 
lemma~\cite[Theorem~3.4,~Section~3]{dinh-jelonek2021};
rather, we use a similar but streamlined approach to prove 
exactly what is needed for the topological decomposition in the
first part of our main results. A robust study of these topics
from the point of view of efficient algorithms is a promising 
direction for future work in the computer algebra community.

\subsection{Present work} \label{ss:present}

We define a set of projective points which we call {\bf stationary
points at infinity (SPAI)}.  These are limits at infinity (or the
coordinate planes) of sequences
%%%RP210218: added coordinate planes here
of points that are asymptotically converging to criticality for
a given height function $h_{\rhat}$.  The ultimate goal is to find 
such sequences on which $h_{\rhat}$ remains bounded, because these
indicate trajectories in which gradient like Morse deformations
may get pulled out to infinity.  Such sequences, together with
limit points of the height function, are called heighted SPAI
(H-SPAI). Their image under $h_{\rhat}$ coincides with the set of 
asymptotic critical values in, e.g.,~\cite{rabier1997,dinh-jelonek2021}.
Note that we use the term \emph{stationary point}
instead of the term critical point (common in the analytic 
combinatorics literature) as `critical point ' 
is overloaded and potentially 
misleading to readers in some areas of mathematics.
%However, the computation of asymptotic critical values is currently
%only tractable when $\rhat$ is rational.  

The advantage to working
with SPAI is that these are easily computed for any real $\rhat$
or when $\rhat$ is a symbolic parameter.  While spurious SPAI
do arise (see Examples~\ref{eg:affine} and \ref{eg:tri}), all such examples we 
know of can be ruled out by determining the height function to be 
unbounded (that is, we do not need to compute the limit set of heights,
just to check whether it is empty).  As to whether H-SPAI themselves
can be spurious, as noted in the introduction of~\cite{dinh-jelonek2021}, 
characterization of bifurcation values (where local triviality fails) is
open, and in particular these can be a proper subset of critical
values.  However, in applications to ACSV, in all examples we know of 
when there are H-SPAI, the attachment cycles do not in fact form a basis 
of the relevant homology group, therefore the isotopy arguments must fail.  

Methodologically, it should be noted that we do not try to show 
topological trivality at infinity in the absence of SPAI,
only that the necessary deformations can avoid infinity.  This,
we suspect, is why our Section~\ref{sec:proof} is shorter
than~\cite[Section~3]{dinh-jelonek2021}.

Our main results are the following.  
\begin{enumerate}[(i)]
\item Given a direction $\rhat$ and real $a < b$, we define three sets
$\spai \supseteq \hspai \supseteq \crit_{[a,b]}$.
\item We give an algorithm for computing SPAI.
\item Theorem~\ref{th:main}, which states that 
cycles may be pushed down until a stationary value is reached in such
a way that they remain above the stationary height only in an arbitrarily small 
neighborhood of the stationary point(s).
\end{enumerate}
As a consequence, when SPAI is empty (which is easily computed) or
when H-SPAI is empty (which may be computed more easily than whether
$\crit_{[a,b]}$ is empty), any cycle may be decomposed into attachment
cycles.  We also note that for a generic $\rhat$ the set SPAI is 
indeed empty, and that computability of SPAI for symbolic $\rhat$
means we can compute a polynomial criterion for the set of directions
$\rhat$ in which SPAI is nonempty.   

The main value of our work lies in its application to ACSV which,
in turn, derives its value from combinatorial applications. The next 
subsection reviews several combinatorial paradigms in which 
ACSV yields strong results;
a different class of examples is presented in Section~\ref{sec:examples}
below.  The purpose of those later examples is $(a)$ to show how 
Theorem~\ref{th:main} can considerably strengthen ACSV analysis, and
$(b)$ to illuminate the role of the hypotheses in the main results
and the increased efficacy in eliminating spurious stationary points.  

Beyond this, Theorem~\ref{th:main} pays back a debt in the literature.
Previous books and papers on ACSV~\cite{PW9,PW-book} often use 
Morse-theoretic heuristics to motivate certain constructions, but 
cannot use Morse theory outright to prove general results.  By 
ruling out stationary points at infinity, those results can be recast 
as following from the stratified Morse framework.  

The remaining sections of the paper, after reviewing applications
of ACSV, are organized as follows.
Section~\ref{sec:critical} sets the notation for the study 
of stratified spaces and stationary points, formulates definitions, and 
states the main result.  Section~\ref{sec:computing} shows how 
to determine all SPAI using a computer algebra system. 
Some examples are given in Section~\ref{sec:examples}.
Section~\ref{sec:proof} proves Theorem~\ref{th:main} by
constructing Morse deformations from height $b$ down to height $a$,
remaining in a bounded region provided $\crit_{[a,b]}$ is empty.

\subsection{Applications of ACSV} \label{ss:applications}

The techniques of analytic combinatorics in several variables find
application to a diverse range of topics in mathematics, computer science,
and the natural sciences. We briefly summarize some of these applications
here; anyone wanting more information can consult, for instance, Pemantle
and Wilson~\cite{PW9,PW-book} or Melczer~\cite{Melczer2021}.  
\\\\
\textbf{Quantum Random Walk:} Since their introduction
in the early 1990s~\cite{AharonovDavidovichZagury1993}, 
quantum variants of random walks have been studied as a 
computational tool for quantum algorithms (see the introduction of Ambainis 
et al.~\cite{AmbainisBachNayakVishwanathWatrous2001}
for a listing of quantum algorithms based around quantum random 
walks, for example).  Results obtained by ACSV go well beyond what
has been obtained by other methods such as orthogonal polynomials or 
the univariate Darboux method~\cite{CarteretIsmailRichmond2003}.
In particular, ACSV may be used to analyze one-dimensional quantum 
walks with arbitrary numbers of quantum 
states~\cite{BresslerGreenwoodPemantlePetkovsek2010} 
and families of quantum random walks on the two-dimensional 
integer lattice~\cite{BaryshnikovBradyBresslerPemantle2011}.
Both of these results involve {\em ad hoc} geometric arguments
which may be streamlined based on the results of the present paper.

\begin{example}
As described in Bressler and Pemantle~\cite{BresslerPemantle2007}, the
analysis of quantum random walks on the one-dimensional integer lattice
can be reduced to studying asymptotics of coefficients
\[ F(x,y) = \frac{G(x,y)}{1-cy+cxy-xy^2} = \sum_{i,j \geq 0} f_{i,j}x^iy^j, \]
where $c \in [0,1]$ is a parameter depending on the underlying probabilities
used to transition between different states in the walk and $G(x,y)$ is a 
polynomial which depends on the initial state of the system. In particular,
for given $c$ one wishes to determine the asymptotic behaviour of the sequence
$a_n^{\lambda} = f_{n,\lceil\lambda n\rceil}$ as $n\rightarrow\infty$.
A short argument about the roots of $H(x,y)=1-cy+cxy-xy^2$ implies 
$a_n^{\lambda}\sim C_{\lambda}n^{-1/2}\rho_{\lambda}^n$ where $0<\lambda<1$;
the values of $\lambda$ such that $\rho_{\lambda}=1$ 
form the \emph{feasible region} of study while the values of $\lambda$ 
with $\rho_{\lambda}<1$, where $a_n^{\lambda}$ exponentially decays, 
form the \emph{nonfeasible region}. For any values of 
$c,\lambda \in (0,1)$ the height function $h_{\mathbb{C}}(x,y)$ with 
$\rhat = (1 : \lambda)$ has two stationary points. Previously, 
to determine asymptotic behaviour one needed to check which of these 
stationary points were in the domain of convergence of $F(x,y)$, 
a computationally difficult task that requires arguing about 
inequalities involving the moduli of variables in an algebraic system 
with parameters.  Running our Maple implementation of Algorithm~\ref{alg:1}
shows that $F(x,y)$ has no stationary points at infinity, meaning 
Theorem~\ref{th:SMT} applies and asymptotics of $a_n^{\lambda}$ 
can be written as an integer linear combination of two explicitly 
known asymptotic series.  In particular, when $2\lambda\in[1-c,1+c]$
then both stationary points are on the unit circle and our results 
immediately imply that  $\lambda$ \emph{must} be in the feasible region. 
Similarly, if $2\lambda\notin[1-c,1+c]$ it can be shown that
$\lambda$ is not in the feasible region. Although our results 
ease the derivation of previously known results
in this instance, they also allow for the derivation of results outside the 
scope of previous methods (see, for instance, Example~\ref{eg:GRZ} below).
\end{example}

\textbf{Queuing Theory and Lattice Walks:} Queuing theory--the study of 
systems in which items enter, exit, and move between various 
lines--arises naturally in computer networking, telecommunications,
and industrial engineering, among other areas. 
Often, one can derive multivariate generating functions describing the
state of a system at point in time, then derive desired information about
the underlying model through an asymptotic analysis. Such analyses,
using analytic combinatorics methods to analyze queuing models, 
can be seen in Bertozzi and McKenna~\cite{BertozziMcKenna1993} and Pemantle and 
Wilson~\cite[Section~4.12]{PW9}, for instance. 
These systems can often be modeled by (classical) 
random walks on integer lattices subject to various 
constraints~\cite[Ch.~9 \& 10]{FayolleIasnogorodskiMalyshev1999}, 
an enumerative problem for which the methods of ACSV are extremely 
effective~\cite{Melczer2021}.
\\\\
\textbf{RNA Secondary Structure:} The secondary structure of 
a molecule's RNA, describing base pairings between its elements, 
encodes important information about the molecule, and predicting 
such structure is a well-studied topic in bioinformatics.  One approach 
to secondary structure prediction uses stochastic context-free grammars 
to generate potential pairings; this approach is implemented in the 
popular Pfold program of Knudsen and Hein~\cite{KnudsenHein1999}. 
To analyze Pfold, Poznanovi{\'c} and Heitsch~\cite{PoznanovicHeitsch2014} 
used multivariate generating functions tracking the probability that 
certain biological features arise. Using classical methods in 
analytic combinatorics, those authors found distributions for single 
features (the numbers of base pairs, loops and helices generated by
a grammar).  Their central limit theorems rely on results of 
Flajolet and Sedgewick, quoted as~\cite[Thoerem~4.1]{PoznanovicHeitsch2014},
whose hypotheses can be replaced by more checkable multivariate hypotheses
once one has our Theorem~\ref{th:main} below.  More recently, 
Greenwood (see~\cite{Greenwood2018} and a forthcoming extension) 
used ACSV to analyze the probability that certain \emph{combinations} of 
features appear.  Greenwood's hypotheses in his Theorem~1 and Corollary~2
can be weakened and much more easily checked with the Morse-theoretic
tools in the present paper.
\\\\
\textbf{Sequence alignment:} The problem of optimally aligning 
more than two sequences on a finite alphabet is fundamental to 
the study of DNA and known in several ways to be mathematically 
intractable.  In~\cite[Section~4.9]{PW9} several cases are analyzed
using techniques of ACSV.  At the time of that paper, ACSV could
only handle cases where the dominant singularity was at a stationary
point all of whose coordinates were known, via Pringsheim's Theorem,
to be real.  Morse theory allows us in principle to handle further 
biologically relevant cases.

\section{Definitions and results} \label{sec:critical}

\subsection{Spaces, stratifications and stationary points} 

Throughout the remainder of the paper, $Q$ is a polynomial and
$\sing$ is the algebraic hypersurface $\{ \zz \in \C^d : Q(\zz) = 0 \}$. 
The elements of $\sing$ with non-zero coordinates is denoted
$\sing_* := \sing \cap \C_*^d$.

\subsubsection{Whitney stratifications}

The following definitions of stratification and Whitney stratification 
are taken from~\cite{Hardt75,GM}.
A {\bf stratification} of a space $\sing \subseteq \C^d$ is a 
partition of $\sing$ into finitely many disjoint sets 
$\{ \str_\alpha : \alpha \in \AA \}$ such that each stratum 
$\str_\alpha$ is a real manifold\footnote{In fact our strata 
are always complex manifolds and complex algebraic sets, however
this is not required in the definition of a stratification.}
of some dimension at most $2d$.
We consider here only algebraic stratifications, meaning that each
stratum is an algebraic set, potentially with an algebraic set
of lower dimension removed.  
A {\bf Whitney stratification} furthermore satisfies the following 
conditions.
\begin{enumerate}[1.]
\item If a stratum $\str_\alpha$ intersects the closure 
$\overline{\str_\beta}$ of another, then it lies entirely
inside: $\str_\alpha \subseteq \overline{\str_\beta}$. 
\item Whitney's Condition B on tangent planes and secant lines
of the strata should hold;  
because we make use only once of this condition and never need to check 
it, we refer readers to~\cite[Chapter~1.2]{GM} for the definition.
\end{enumerate}

A stratification of the pair $(\C_*^d , \sing_*)$ is a Whitney stratification 
of $\C_*^d$ in which the only $(2d)$-dimensional stratum is $\M := 
\C_*^d \setminus \sing$.  When $\sing$ is a complex algebraic variety,
such a stratification always exists. 

\subsubsection*{Logarithmic space, natural Riemannian metric, tilde notation}

Let $\T := \R / 2 \pi \Z$ denote the one-dimensional torus.
We introduce the {\em logarithmic space (logspace)} $\logspace\cong 
\T^d\times\R^d$ together with the exponential diffeomorphism
$\exp : \logspace \to \C_*^d$ defined for $\eta \in \T^d$ and $\xi \in \R^d$ by
$$\exp(\eta,\xi) := \exp(\xi+i\eta).$$
We refer to $\xi \in \R^d$ as the {\em real} coordinates on the logspace.
Mostly we do not need logspace until Section~\ref{sec:proof}, though
we refer to coordinates $\xi$ and $\eta$ occasionally.  The reason for
using both $\C_*^d$ and the logspace is that deformations and other
geometric constructions are more transparent in $\logspace$ but 
that polynomial computations via computer algebra involving $Q$ must be 
carried out in $\C_*^d$.  

\subsubsection*{Directions and affine stationary points}

A {\em direction} is an equivalence class of vectors in 
$\R^d$ under positive multiples.  The direction containing
the nonzero vector $\rr$ can be canonically identified with
the unit vector $\rhat$.  Fix a {\em direction} $\rhat$.
The phase function $h_{\rhat}$ from~\eqref{eq:phase} is
given in logspace by the linear function $\rhat \cdot \xi$.
stationary points of $h_{\rhat}$ on any complex-analytic submanifold 
of $\C_*^d$ are therefore the same as stationary points of
a branch of $\zz^\rr$.

Fix a stratification $\{ \str_\alpha \}$ of $\sing$.
For a stratum $\str$, we define an {\em affine $\rr$-stationary point} 
as a stationary point of the restriction of $h_{\rhat}$ to $\str$. 
These points are what is referred to in ACSV literature as 
critical points; we call them {\em affine stationary points},
in contrast to SPAI, which are stationary points at infinity. 

\subsubsection*{Stationary point relation}

Given a stratum $\str$, a point $\xx \in \str$ and a direction $\rhat$,
a necessary and sufficient condition for $h_{\rhat} |_\str$ to have
a stationary point at $\xx$ is a drop in the rank of a certain matrix of
differentials.  To make this more precise, choose a 
neighborhood of $\xx$ in which the closure of $\str$ is locally 
cut out by $c$ independent polynomials, where $c$ is the 
codimension of $\str$,
\begin{equation} \label{eq:stratum}
\overline{\str} = \left\{\zz : f^\str_j (\zz) = 0 : 
   1 \leq j \leq c \right\} \, .
\end{equation}
The stratum $\str$ may be obtained by intersecting $\overline{\str}$ with a set 
$\{ g^\str_i (\zz) \neq 0 \}$ of polynomial non-equalities 
to remove points in substrata.
By refining the stratification if necessary, we can assume that
the polynomials $f^\str_j$ and $g^\str_i$ define the
whole stratum and that the differentials $\{ df_j^\str \}$ 
are linearly independent everywhere on $\str$.
For any point $\zz \in \str$ and vector $\yy \in \C^d$, let
$\J(\zz,\yy) = \J^{\str , \rhat} (\zz , \yy)$ 
denote the $(c+1) \times d$ matrix with entries
\begin{equation}
\J_{i,j} := z_j \frac{\partial f_i}{\partial z_j} \mbox{ for }
1 \leq i \leq c, \mbox{ and } \J_{c+1, j} = y_j \, .
\label{eq:Jmat}
\end{equation}
Note that after dividing the $j$th column of $\J(\zz,\yy)$ by $z_j$
its rows become the gradients $(\nabla f_i)(\zz)$ together with the 
scaled gradient $(\nabla h_{\yy})(\zz)/h_{\yy}(\zz)$.
The rank of $J(\zz,\ww)$ is invariant under the map $\ww \mapsto \lambda \ww$
for any non-zero $\lambda$, thus we define a binary 
relation $(\zz,\yy)\in S$ which we interpret as $(\nabla h_{\yy})(\zz)$ 
lying in the normal space to $\str$ at $\zz$.

\begin{definition}[binary relation for stationary points]
Define $S = S(\str) \subseteq \C_*^d \times \CP^{d-1}$ by $(\zz,\yy) \in S$
if and only if $\rk(J(\zz,\yy)) \leq c$.
\end{definition}

\begin{proposition}
Let $\str$ be a stratum and suppose $\zz \in \str$ and $\rhat$ is real.
then $(\zz , \rhat) \in S(\str)$ if and only if $\zz$ is a critical point
of the height function $h_{\rhat}$ on the stratum $\str$.
\end{proposition}

\begin{proof}
The tangent space to a stratum $\str$ defined by analytic functions
is a complex linear subspace of $\C^d$.  The function $h_{\rhat}$ is
the real part of a (branched) analytic function
$h^{\C}_{\rhat} := \sum_{j=1}^d r_j \log z_j$.  It follows that the
vanishing of $dh_{\rhat} |_\str$ on the real part of the complex tangent
space $T \str$ is equivalent to its vanishing on all of $T \str$, which
is equivalent to the vanishing of $dh^{\C}_{\rhat}$ on all of $T \str$.
Vanishing of the complexified height is equivalent to $\rhat$ being in
the complex normal space, which is equivalent to the rank being at most $c$
of any basis for the normal space, together with the vector $\rhat$.
\qed
\end{proof}

\subsubsection*{Stationary points at infinity} 

\begin{definition}[SPAI] \label{def:SPAI}
Let $\overline{S}$ denote the closure of $S$ when embedded in 
$\CP^d \times \CP^{d-1}$.  Define a SPAI in direction $\rhat$
to be an element $(\zz , \rhat) \in \overline{S}$ where $\zz \notin \C_*^d$,
meaning $\zz$ either lies in the plane at infinity or has at least 
one vanishing coordinate.  A {\em witness} to the SPAI $(\zz , \rhat)$
is a sequence $(\zz_n , \rhat_n)$ in $\C_*^d \times \CP^{d-1}$ 
converging to $(\zz , \rhat)$.
\end{definition}

\begin{definition}[ternary relation for heighted stationary points]
\label{def:graph}
Fix a direction $\rhat$.
Let $R := R(\str , \rhat)$ denote the set of triples 
$(\zz , \yy , \eta) \in \C_*^d \times \CP^{d-1} \times \C$ 
such that the following three conditions hold.  
\begin{enumerate}[(i)] \itemsep0em
\item $\zz \in \str$ ; 
\item $\rk(\J (\zz,\yy)) \leq c$, where $c$ is the co-dimension of $\Sigma$ ;
\item $h_{\rhat} (\zz) = \eta$.
\end{enumerate}
Projecting $R(\str, \rhat)$ to the first two coordinates yields $S(\str)$.
\end{definition}

\begin{definition}[H-SPAI] \label{def:crit}
Let $\overline{R}$ denote the closure of $R$ in $\CP^d \times \CP^{d-1}
\times \R$.  A triple $(\zz , \yy, \eta) \in \overline{R}$ with
$\zz \notin \C_*^d$ is called an H-SPAI in $\str$ in direction $\rhat$
and is said to have height $\eta$.  A {\em witness} for the H-SPAI
$(\zz , \yy , \eta)$ is a sequences $(\zz_n , \yy_n , \eta_n)$ in
$R(\str , \rhat)$ converging to $(\zz , \yy, \eta)$.
\end{definition}

We say that the real number $\eta$ is a {\em generalized stationary value}
of $h_{\rhat}$ on the stratum $\str$ if either it is an affine stationary
value (that is, a stationary value of $h_{\rhat} |_\str$) or else it
is the third coordinate of some H-SPAI.  We denote the set of stationary
values by $K(\str , \rhat) = K_0 \cup K_\infty$.

\begin{definition}[stationary points in an interval] \label{def:notation}
Fix $\rhat, \str,$ and fix $-\infty \leq a < b \leq \infty$.  
The $\zz$ coordinates of 
all stationary points on $\str$ with heights in $[a,b]$ form the set
$$\crit_{[a,b]} (\str , \rhat) := \{ \zz \in \CP^d : \exists 
   \yy , \eta \mbox{ with } (\zz , \yy , \eta) \in \overline{R}, \yy = \rhat 
   \mbox{ and } \eta \in [a,b] \} \, .$$
Omitting the argument $\str$ or $[a,b]$ denotes 
a union over all strata and taking $[a,b] = (-\infty , \infty)$, respectively.
We write $\crit_{[a,b]}^{\rm aff}(\str , \rhat)$ for the elements of
$\crit_{[a,b]} (\str , \rhat)$ which are affine stratified points and
$\crit^\infty_{[a,b]} (\str , \rhat)$ for the remaining elements.
\end{definition}

\begin{remark}
Because we sometimes need to refine stratifications, we note that 
refining the stratification can introduce more stationary points, 
affine or infinite, but cannot remove any.  
\end{remark}

\subsection{Main topological results}

For any space $S$ with height function $h$ and any real $b$, 
we define $S_{\leq b} = \{ \xx \in S : h(\xx) \leq b \}$.
We first state our main deformation result, an extension to
the nonproper setting of well known stratified Morse theoretic results
for proper height functions.

\begin{theorem}[Morse deformation] 
\label{th:main}
Fix $Q, \sing$, $\M := \C_*^d \setminus \sing$, and a Whitney 
stratification $\{ \str_\alpha \}$ of $(\C_*^d , \sing_*)$.
Fix also a direction $\rhat$ and height function $h_{\rhat}(\zz)=\zz^{\rhat}$.  
\begin{enumerate}[(i)] 
\item Suppose $\crit_{[a,b]} (\str , \rhat)$ is empty.  Then
$\str_{\leq b} \cong \str_{\leq a}$
for any stratum $\str$, and $\M_{\leq b} \cong \M_{\leq a}$.
\item Suppose $\crit_{[a,b]} (\str , \rhat) = \crit_{[a,b]}^{\rm aff}
(\str , \rhat) = \{ \zz_1, \ldots , \zz_k \}$ with $h_{\rhat} (\zz_j) = c
\in (a,b)$ for all $j \leq k$.  Then for any stratum $\str$, any chain 
$\CC$ supported on $\str_{\leq b}$ is homotopic in $\str_{\leq b}$ 
to a chain supported on the union of $\str_{< c}$ together with 
arbitrarily small neighborhoods of the points $\zz_k$ in $\str$.  
Taking $\str = \M$, it follows that every homology class in 
$H_d (\M_{\leq b})$ is represented by a cycle supported on this union.
\end{enumerate}
\end{theorem}

\begin{remark}
This theorem draws no conclusion about whether the topology 
of $\M$ or $\sing$ can be deduced from the topology near the
stationary points, only stating that certain needed deformations 
exist.  In fact, all the topological information necessary to 
estimate the Cauchy integrals is present in the relative 
homology group $(\M , \M_{\leq -K})$ for some sufficiently 
large $K$, hence the topology of $\M$ at sufficiently low
heights is irrelevant.
\end{remark}

The purpose of these homotopy equivalences is to push the
cycle of integration $T$ down to one whose maximum height 
is as low as possible.  For example,
because the cycle $T$ in the Cauchy integral of interest
can be pushed down at least until hitting 
the first stationary point corresponding to direction $\rhat$, 
the magnitude of coefficients in direction $\rhat$ is bounded above 
by the Cauchy integral over a contour at this height.  The
following corollary of Theorem~\ref{th:main} was given only 
as a conjecture in~\cite{PW-book} because it was not known 
under what conditions $T$ could be pushed down to the stationary height.  

\begin{corollary} \label{cor:main}
Fix $\rhat$ and a Laurent polynomial $Q$ and Laurent expansion 
$P(\zz)/Q(\zz) = \sum_{\rr \in K} a_\rr \zz^\rr$.  Suppose 
$K(\str , \rhat)$ is finite for all strata and denote the 
maximum stationary value by $\disp c$. Then
$$\limsup_{\substack{\rr \to \infty \\ \rr / |\rr| \to \rhat}} 
   \frac{1}{|\rr|} \log |a_\rr| \leq c \, .$$
\end{corollary}

More generally, we would like to examine how a cycle $\CC$ in $\M$ 
can be represented as the sum of cycles, each of which has been 
pushed down until reaching an obstacle at some stationary height.
In the case where $h_{\rhat}$ is proper, this is a classical result
of Morse theory (when $\sing$ is smooth) or more generally of
stratified Morse theory.  We briefly recall the relevant Morse
theoretic notions.

Let $h: X \to \R$ be a proper smooth function on a stratified space $X$.
Suppose that $h$ has finitely many (stratified) stationary points
$\xx_1, \ldots , \xx_k$.  For ease of exposition, assume the stationary
values $c_1 := h(\xx_1) > \cdots > c_k := h(\xx_k)$ are distinct.
Let $(X_j^+ , X_j^-)$ denote the pair 
$(X_{\leq c_j - \ee} \cup B_{2 \ee} (x_j) , X_{\leq c_j - \ee})$ 
where $\ee$ is sufficiently small.  The relative homology group
$H_d (X_j^+ , X_j^-)$ is called the {\em attachment group} at $x_j$.
Stratified Morse theory guarantees that the attachment pairs
generate all the homology of $X$.  For example, the
following proposition is well known 
(see, e.g.,~\cite[Section~B.2]{PW-book}).

\begin{proposition}[attachments generate homology: proper case]
\label{pr:attachments}
~~\\[-3ex]
\begin{enumerate}[(i)]
\itemsep0em
\item
Every integer homology class $\CC \in H_d (X)$ can be written as a finite
integer combination of classes $\alpha \in H_d (X_j^+)$ projecting
to a nonzero class in $H_d (X_j^+ , X_j^-)$.  
\item
Let $G_j$ be the image of $H_d (X_j^+)$ under the projection from
$X_j^+$ to $(X_j^+ , X_j^-)$, that is, those relative homology classes
representable by absolute cycles.  If $d$ is the homological 
dimension of $X$ then $H_d (X) \cong \bigoplus_{j=1}^k G_j$.
\item When $X$ is a smooth $2d$-manifold and $h$ is harmonic, 
each attachment group is a homology $d$-sphere,
each $G_j$ is is the whole attachment group $H_d (X_k^+ , X_k^-)$ and
$H_d (X) \cong \bigoplus_{j=1}^k H_d (X_k^+ , X_k^-)$.
\end{enumerate}
\end{proposition}

\begin{proof}
The first statement follows from the homotopy equivalence between
$X_k^+$ and $X_{k-1}^+$ (because Theorem~\ref{th:main} always holds
for proper height functions) and the standard Morse filtration.
The second statement follows from the long exact sequence for the
pair $(X_k^+ , X_k^-)$ and the vanishing of $H_{d+1} (X_k^-)$.
The third follows because attachments in smooth Morse theory
are $d$-balls modulo their boundaries, where $d$ is the index
of the stationary point.
\qed
\end{proof}

\begin{remark}
When the stationary values are not distinct one may add in the
balls $B_{2\ee} (x_j)$ one at a time, arriving at the same result.  
\end{remark}

Our main topological result is that all of this holds for the 
nonproper height function $h_{\rhat}$ unless obstructed by 
stationary points at infinity.  

\begin{theorem}[attachments generate homology: general case] \label{th:main2}
Fix $Q, \sing$, $\M := \C_*^d \setminus \sing$, and a Whitney
stratification $\{ \str_\alpha \}$ of $(\C_*^d , \sing_*)$.
Fix also a direction $\rhat$ and 
let $\zz_1 , \ldots \zz_k$ be the affine stationary points
with $c_j := h_{\rhat} (\zz_j)$ nonincreasing.  Let $\M_j^\pm$
be the spaces $X_k^\pm$ from above, with $X = \M$.
\begin{enumerate}[(i)] 
\item Suppose $K_\infty = \emptyset$, that is, there are no H-SPAI.  
Then $H_d (\M) \cong \bigoplus_{j=1}^k G_j$ where $G_j$ are the 
relative cycles in $H_d (\M_j^+ , \M_j^-)$ represented by absolute cycles.
\item Suppose $K_\infty$ is nonempty, having maximum element $c$.
Let the $\{ c_j \}$ be the stratified values as above, 
with $s$ chosen so that 
$c_s > c \geq c_{s+1}$.  Then any class $\CC \in H_d (\M)$ may be
written as $\CC = \beta + \sum_{j=1}^s \alpha_j$ where $\alpha_j \in G_j$
for $j \leq s$ and $\beta$ is supported on $\M_{\leq c + \ee}$.
\end{enumerate}
\end{theorem}

\begin{proof}
Part~$(i)$ follows from Theorem~\ref{th:main} and 
Proposition~\ref{pr:attachments} by deforming each
$\M_j^-$ to $\M_{j+1}^+$.  Part~$(ii)$ follows by performing the
deformations only down to $j = s$.
\qed
\end{proof}

Theorem~\ref{th:main2} says that in Steps I--III of the asymptotic coefficient 
evaluation from Section~\ref{ss:motivation} the only integrals
that need to be evaluated are integrals over classes in each $G_j$.
In general, these attachment classes are the best places to integrate.
For example, as mentioned above, if a stationary point $\xx$ is a smooth 
point of $\sing_*$ then $G$ will have a single generator which is a saddle
point contour and an asymptotic expansion can be deduced automatically.

We also remark that the direct sum in part~$(i)$ and the expansion of
$\CC$ in part~$(ii)$ are not natural.  Elements of $G_j$ are equivalence
classes modulo $G_i$ for all $i > j$, and correspondingly $\alpha_j$
in part~$(ii)$ is determined only modulo linear combinations of $\alpha_i$
for $i > j$.  However, the pair $(j_* , \alpha_{j_*})$ is well defined, 
where $j_*$ is the least index $j$ for which $\alpha_j \neq 0$ (in part~$(ii)$,
the least among $1, \ldots , s$). 

In the remainder of Section~\ref{sec:critical} we expand on 
our underlying motivations, describing the application of
Theorems~\ref{th:main} and~\ref{th:main2} to the computation of cycles, 
integrals over these cycles, and coefficient asymptotics of multivariate
rational functions.  These results, collected from
various prior and simultaneous works, can be skipped if one is
only interested in examples, proofs and computations of stationary 
points at infinity.  

\subsection{Intersection classes on smooth varieties} \label{ss:intersection}

We assume throughout this section that $\sing$ is smooth.
It is useful to be able to transfer between $H_d (\M)$ and 
$H_{d-1} (\sing_*)$: topologically this is the Thom isomorphism and,
when computing integrals, corresponds to taking a single residue.  
We outline this construction, which goes back at least 
to Griffiths~\cite{griffiths1969}.  
Because $\grad Q$ does not vanish on $\sing$, the well known 
Collar Lemma~\cite[Theorem~11.1]{milnor-stasheff} 
states\footnote{See~\cite{lang-manifolds} for a full proof.} that an 
open tubular vicinity of $\sing$ is diffeomorphic to the space 
of the normal bundle to $\sing$.  

It follows that for any $k$-chain $\gamma$ in $\sing$ 
we can define a $(k+1)$-chain $\leray \gamma$, obtained
by taking the boundary of the union of small disks in the fibers
of the normal bundle.  The radii of these disk should be small
enough to fit into the domain of the collar map, but can
(continuously) vary with the point on the base.  Different choices
of the radii matching over the boundary of the chain lead to homologous
tubes.  We will be referring to $\leray \chain$ informally as the
{\em tube around $\gamma$}.  Similarly, the symbol $\disk \chain$
denotes the product with the solid disk.  The elementary rules for
boundaries of products imply
\begin{equation} \label{eq:bdry disk}
\begin{array}{rcl}
\partial (\leray \gamma) & = & \leray (\partial \gamma) \, ; \\
\partial (\disk \gamma) & = & \leray \gamma \cup \disk (\partial \gamma) \, .
\end{array}
\end{equation}
Because $\leray$ commutes with $\partial$, cycles map to cycles,
boundaries map to boundaries, and the map $\leray$ on
the singular chain complex of $\sing_*$ induces a map on 
homology $H_* (\C_*^d \setminus \sing)$; we also denote this 
map on homology by $\leray$ to simplify notation.

\begin{proposition}[intersection classes] \label{pr:int}
Suppose $\sing := \{ Q=0 \}$ is smooth, and define 
$\leray : H_{d-1} (\sing_*) \to H_d (\M)$ as above.
\begin{enumerate}[(i)]
\item 
$\circ$ is injective and its image is the kernel of the
map $\iota_*$ induced by the inclusion 
$\M {\stackrel{\iota}{\longrightarrow}} \C_*^d$.
\item 
Given $\alpha \in \ker (\iota_*)$, one may compute the pullback
$\I (\alpha) := \leray^{-1} (\alpha)$ by intersecting $\sing_*$ 
with any $(d+1)$-chain in $\C_*^{d+1}$ whose boundary is $\alpha$,
and for which the intersection with $\sing_*$ is transverse.
\end{enumerate}
Specializing to $\alpha = \TT - \TT'$ where $\TT$ and $\TT'$ are two
$d$-cycles in $\M$ homologous in $\C_d^*$, we call $\I (\TT - \TT')$
the {\bf intersection class} of $\TT$ and $\TT'$.
\end{proposition}

\begin{proof}
The Thom-Gysin long exact sequence implies exactness in the following
diagram,
\begin{equation}\label{eq:thom}
H_{d+1}(\Comp_*^d) \stackrel{I_*}{\to} H_{d-1}(\sing_*) 
   \stackrel{\leray}{\to} H_d(\M) \to 
   H_d(\Comp_*^d).
\end{equation}
This may be found in~\cite[page~127]{gordon-residues}, taking
$W = \C_*^d$, though in the particular situation at hand it goes 
back to Leray~\cite{leray}; here, the first mapping, $I_*$, denotes 
the map induced by transverse intersection, $I$.  Injectivity of 
$\leray$ follows from the vanishing of $H_{d+1} (\C_*^d)$.  
The rest of part~$(i)$ follows from exactness at $H_d (\M)$.

For part~$(ii)$, we begin by showing
that $I$ induces a well defined map from $\ker (\iota_*)$ to 
$H_{d-1} (\sing_*)$.  Given $\alpha \in \ker (\iota_*)$,
because transversality is generic, there exist $(d+1)$-chains 
intersecting $\sing_*$ transversely whose boundary is $\alpha$.  
If $\D$ is such a chain and $\CC = I(\D)$ then $\CC$ is a cycle:
$$\partial \CC = \partial (\D \cap \sing_*) = (\partial \D) \cap \sing_*
   = \alpha \cap \sing_* = \emptyset \, .$$
Let $\D_1$ and $\D_2$ be two such chains, and denote $\CC_j 
:= \D_j \cap \sing_*$.  Observe that $\D_1 - \D_2$ is null homologous 
because there is no $(d+1)$-homology in $\C_*^d$, whence 
$$[\CC _1 - \CC_2] = [I(\D_1 - \D_2)] = 0 \, ,$$
showing that $[I(\D)]$ for $\partial \D = \alpha$ is 
well defined in $H_{d-1} (\sing_*)$.  

Finally, if $\alpha = \leray (\gamma)$ then taking $\D = \disk (\gamma)$
gives $I(\D) = \gamma$, showing that $I$ does in fact invert $\leray$,
hence computes $\I$.
\qed
\end{proof}

\subsection{Integration}

Integrals of holomorphic forms on a space $X$ are well defined 
on homology classes in $H_* (X)$.  Relative homology is useful 
for us because it defines integrals up to terms of small order.
Throughout the remainder of the paper, $F = P/Q$ denotes a
quotient of polynomials except when a more general numerator is
explicitly noted.  Let $\amoeba (Q)$ denote the amoeba 
$\{ \log |\zz| : \zz \in \sing_* \}$ associated to the 
polynomial $Q$, where $\log$ and $| \cdot |$ are taken 
coordinatewise\footnote{One should think of the amoeba as sitting
in $\xi$-space, the real part of logspace.}.
Components $B$ of the complement of 
$\amoeba (Q)$ are open convex sets and are in correspondence
with convergent Laurent expansions $\sum_{\rr \in E}
a_\rr \zz^\rr$ of $F(\zz)$, each expansion being convergent when 
$\log |\zz| \in B$ and determined by the Cauchy integral~\eqref{eq:cauchy}
over the torus $\log |\zz| = \xx$ for any $\xx \in B$.

\begin{definition}[$c_*$ and the pair $(\M , -\infty)$] \label{def:low}
Fix $\rhat_*$ and let $c_*  = c_* (\rhat_*)$ denote 
the infimum of heights of stationary points, both affine and at infinity.
Denote by $H_d (\M , -\infty)$ the homology of the pair 
$(\M , \M_{\leq c})$ for any $c < c_*$.  By part~$(i)$ of 
Theorem~\ref{th:main}, these pairs are all naturally 
homotopy equivalent.
\end{definition}

For functions of $\rr \in E \subseteq (\Z^+)^d$, let $\expdecay$ 
denote the relation of differing by a quantity decaying more rapidly 
than any exponential function of $|\rr|$.  If $E$ consists of
vectors $\rr$ whose angle with a fixed $\rr_*$ is bounded above by 
$\pi/2 - \ee$, we note for use below that $h_{\rhat} \leq 
\ee h_{\rhat_*}$, in other words, $h_{\rhat}$ and $h_{\rhat_*}$
go to $-\infty$ at comparable rates on $E$.
Homology relative to $-\infty$ and equivalence up to superexponentially
decaying functions are related by the following result.

\begin{theorem} \label{th:low}
Let $F = G/Q$ with $Q$ rational and $G$ holomorphic.  Fix $\rhat_*$
and suppose that $c_* (\rhat_*) > -\infty$.  For $d$-cycles $C$ in $\M$, 
as $\rr$ varies over a set $E$ whose angle with $\rhat_*$ is bounded
above by $\pi - \ee$, the $\expdecay$ equivalence class of the integral 
$\disp \int_{\CC} \zz^{-\rr} F(\zz) \, d\zz$
depends only on the relative homology class of $\CC$ when projected
to $H_d (\M , -\infty)$.  
\end{theorem}

\begin{proof}
Fix any $c < c_*$.  Suppose $C_1 = C_2$ in 
$H_d (\M , -\infty)$.  From the exactness of
$$H_d (\M_{\leq -c}) \to H_d (\M) \to  H_d (\M , \M_{\leq c}) \, ,$$
observing that $C_1 - C_2$ projects to zero in $H_d (\M , \M_{\leq c})$,
it follows that $C_1 - C_2$ is homologous in $H_d (\M)$ to some cycle
$C \in \M_{\leq c}$.  Homology in $\M$ determines the integral
exactly.  Therefore, it suffices to show that $\int_C \zz^{\rr} 
F(\zz) \, d\zz \expdecay 0$.

As a consequence of the homotopy equivalence in part~$(i)$ of 
Theorem~\ref{th:main}, for any $t < c_*$ there is a cycle 
$C_t$ supported on $\M_{\leq t}$ and homologous to $C$ in $\M$.
Fix such a collection of cycles $\{ C_t \}$.  Let
$M_t := \sup \{ |F (\zz)| \, : \, \zz \in C_t \}$ and let
$V_t$ denote the volume of $C_t$.  Observe that $|\zz^{-\rr}|
= \exp (|\rr| h_{\rhat} (\zz)) \leq \exp (\ee |\rr| h_{\rhat_*} (\zz)) 
\leq \exp (\ee t |\rr|)$ on $C_t$.  It follows that
\begin{eqnarray*}
\int_C \zz^{-\rr} F(\zz) \, d\zz & = & \int_{C_t}  \zz^{-\rr} F(\zz) \, d\zz \\
& \leq & V_t M_t \exp (t |\rr|).
\end{eqnarray*}
Because this inequality holds for all $t < c_*$, the integral is
thus smaller than any exponential function of $|\rr|$.
\qed
\end{proof}

When $F = P/Q$ is rational we may strengthen determination up to 
$\expdecay$ to exact equality for all but finitely many coefficients.
The {\bf Newton polytope}, denoted 
$\newpol$, is defined as the convex hull of degrees $\mm \in \Z^d$ 
of monomials in $Q$.  
It is known (see, e.g.,~\cite{forsberg-passare-tsikh}) that 
the components of $\amoeba(Q)^c$ map injectively into the 
integer points in $\newpol$, and that to each extreme point 
$\newpol$ corresponds a non-empty component.  Moreover, this
can be done in such a way that the 
recession cone of a component (collection of directions 
of rays contained in the component) equals the 
dual cone of the Newton polytope at the corresponding vertex;
see Figure~\ref{fig:amoeba}.  
Hence the linear function $\xi \mapsto -(\xi \cdot \rr)$ 
is unbounded from below on any component
when $\rr$ points in the same direction as any element of the dual cone 
of the Newton polytope $\newpol(Q)$ at the corresponding integer point. 
Fix the component $B$ corresponding to the Laurent expansion
$F = \sum_{\rr} a_\rr \zz^\rr$ and integer point $\mathbf{v}$ in the 
Newton polytope.  The closed dual cones at
extreme points of the Newton polytope cover all of $\R^d$, therefore
there exists a component $B'$ of the amoeba complement
(probably many components would do) with $h_{\rhat} (\xi) \to -\infty$ 
linearly in $|\xi|$ as $\xi \to \infty$ in $B'$.

\begin{figure}
\centering
\includegraphics[width=0.35\linewidth]{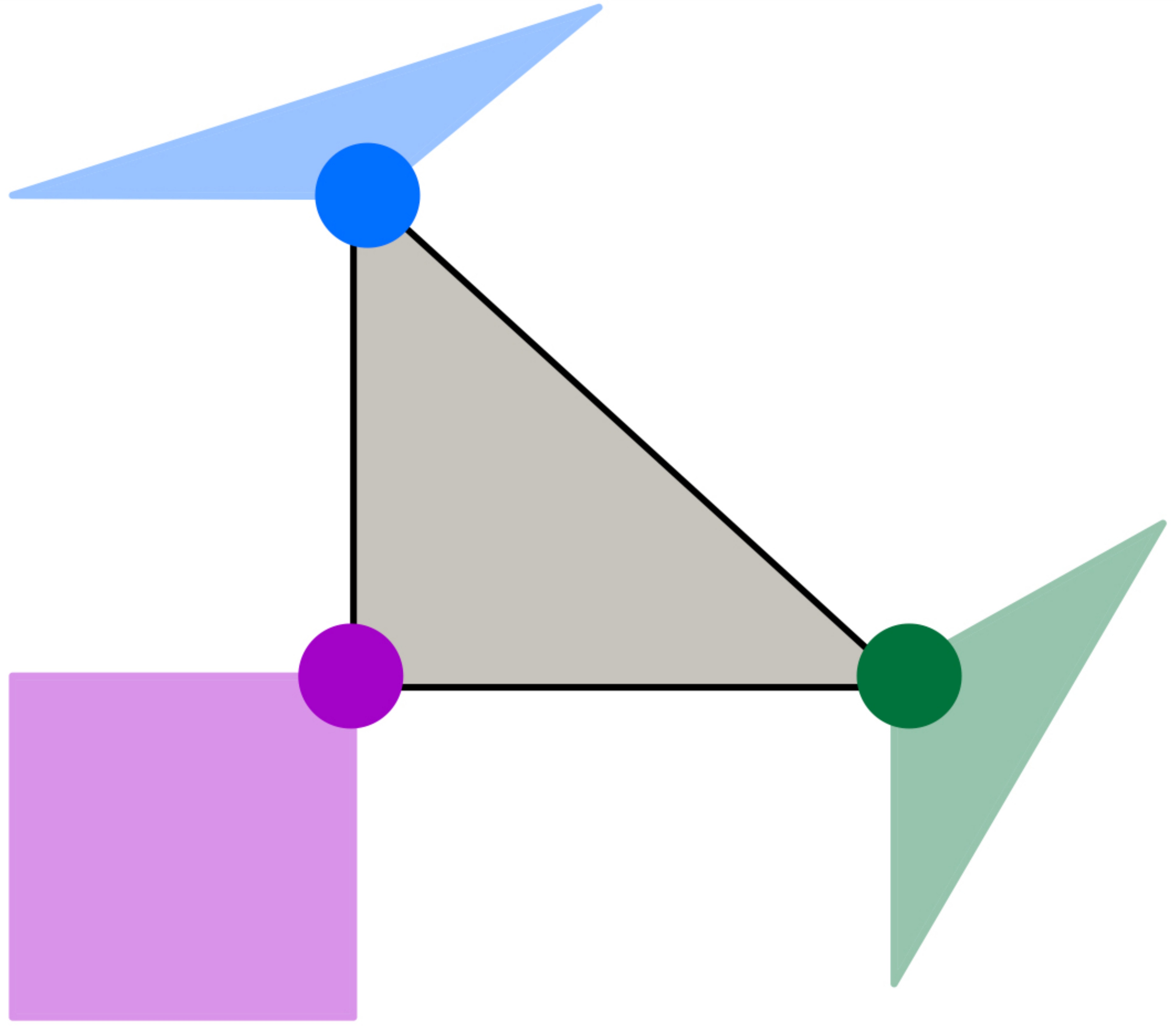}
\qquad\qquad\qquad
\includegraphics[width=0.4\linewidth]{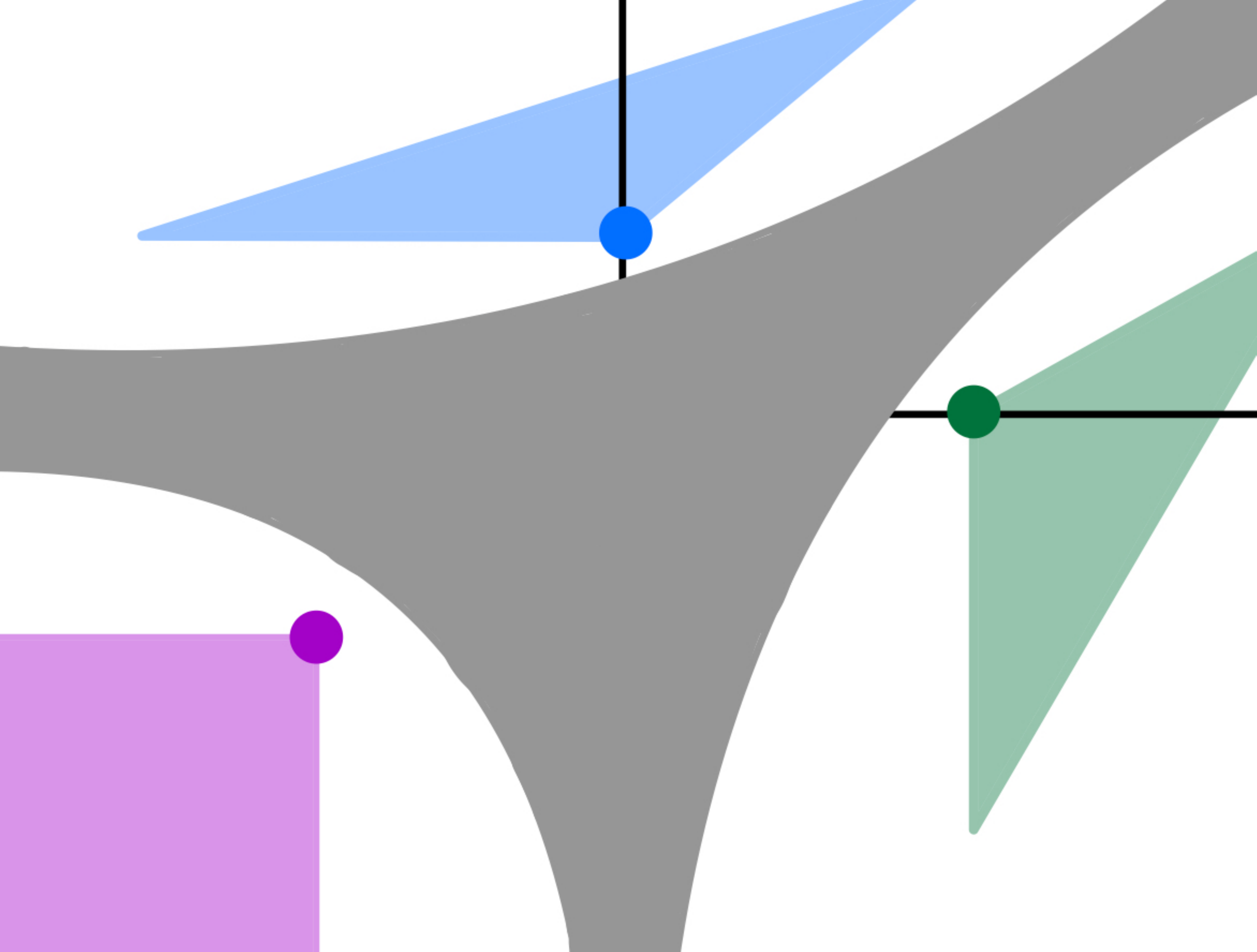}
\caption{\emph{Left:} The Newton polytope of $Q(x,y) = 1-x-y$ together with the dual cones
at each vertex. \emph{Right:} The amoeba of $Q(x,y)$ together with the recession cones of
the complement components. }
\label{fig:amoeba}
\end{figure}

\begin{proposition} \label{pr:vanish}
Let $\TT (\xi)$ denote the centered torus with polyradii $\exp (\xi_1) , 
\ldots , \exp (\xi_d)$.  If $F = P/Q$ is rational and $\xi \in B'$, then
$$\int_{\TT (\xi)} \zz^{-\rr} F(\zz) \, d\zz = 0$$
for all but finitely many $\rr \in E$, the support of the 
Laurent expansion on $B$.
\end{proposition}

\begin{proof}
By assumption, there is a continuous path moving $\xi$ to infinity 
within $B'$.  On the corresponding tori, the (constant) value 
of $\hr$ approaches $-\infty$.  Let $\TT (\xi_t)$ denote such a
torus supported on $\M_{\leq t}$.  Because the tori are all homotopic 
in $\M$, the value of the integral 
\begin{equation} \label{eq:TT_t}
\int_{\TT (\xi_t)} \zz^{-\rr} F(\zz) \, d\zz
\end{equation}
cannot change.  On the other hand, with $M_t$ and $V_t$ as in the
proof of the first part, both $M_t$ and $V_t$ are bounded by
polynomials in $|\zz|$, the common polyradius of points in $\TT (\xi_t)$.
Once any coordinate $r_j$ is great enough so that the product
of the volume and the maximum grows more slowly than $|z_j|^{r_j}$
the integral for that fixed $\rr$ goes to zero as $t \to -\infty$,
and thus is identically zero.
\qed
\end{proof}

The utility of Proposition~\ref{pr:vanish} is to represent the Cauchy integral
as a tube integral.  Let $\TT = \TT (\xi)$ for $\xi \in B$, the
component of $\amoeba(Q)^c$ defining the Laurent expansion, and 
choose $\TT' = \TT (\xi')$ for $\xi' \in B'$ as
in Proposition~\ref{pr:vanish}.  By Proposition~\ref{pr:int}, if
$\gamma$ denotes the intersection class $\I (\TT , \TT')$, we have
$\TT = \leray \gamma + \TT'$ in $H_d (\M)$.  By Proposition~\ref{pr:vanish}, 
the integral over $\TT'$ vanishes for all but finitely many $\rr \in E$,
yielding

\begin{corollary} \label{cor:tube}
If $F = P/Q$ is rational and $\sing$ is smooth, then 
there exists a $(d-1)$-cycle of integration $\gamma$
in $\sing_*$ such that for all but finitely
many $\rr \in E$, 
$$ a_\rr = \frac{1}{(2 \pi i)^d}\int_{\leray \gamma} \zz^{-\rr - \one} F(\zz) \, d\zz
   \, .$$
\qed
\end{corollary}

\subsection{Residues on smooth varieties} \label{ss:residues}

This section again assumes that $\sing_*$ is smooth.
Having transferred homology from $\M$ to $\sing_*$ via intersection
classes, we transfer integration there as well via residues.  
The point of this is to obtain integrals amenable to a saddle point 
analysis: pushing down cycles until their maximum height is minimized 
drives the maximum to occur on $\sing_*$, not on $\M$.  
Thus we need a reduction to saddle point integrals on $\sing_*$
rather than on $\M$.  In what follows, $H^* (X)$ denotes the 
holomorphic de Rham complex, whose $k$-cochains are
holomorphic $k$ forms.  The following duality between 
residues and tubes is well known.

\begin{proposition}[residue theorem] \label{pr:residue}
There is a functor $\Res : H^d (\M) \to H_{d-1} (\sing_*)$ such 
that for any class $\gamma \in H_d (\sing)$ and every $\omega \in
H^d (\M)$,
\begin{equation} \label{eq:residue}
\int_{\leray \gamma} \omega = 2 \pi i \, \int_{\gamma} \Res (\omega) \, .
\end{equation}
The residue functor is defined locally and, when $Q$ is squarefree,
it commutes with products by any locally holomorphic scalar function.  
If, furthermore, $F = P/Q$ is rational, there is an implicit formula
$$Q \wedge \Res (F \, d\zz) = P \, d\zz \, .$$
For higher order poles, the residue can be computed by choosing 
coordinates: if $F = P / Q^k$, and locally $\{Q=0\}$ defines a graph of a function,
$\{z_1=S(z_2,\ldots,z_d)\}$, 
then
\begin{equation} \label{eq:residue k}
\Res_\sing \left [ \zz^{-\rr} F(\zz) \frac{d\zz}{\zz} \right ]
   := \left.\frac{1}{(k-1)! (\partial Q / \partial z_1)^k}
   \frac{d^{k-1}}{dz_1^{k-1}} \left [
   \frac{P \zz^{-\rr}}{\zz} \right ]\right\vert_{z_1=S(z_2,\ldots,z_d)} \, dz_2 \wedge \cdots \wedge dz_d \, .
\end{equation}
\end{proposition}

\begin{proof}
Restrict to a neighborhood of the support of the cycle $\gamma$ in the
smooth variety $\sing_*$ coordinatized so that the last coordinate
is $Q$.  The result follows by applying the (one variable) residue 
theorem, taking the residue in the last variable.
\qed
\end{proof}

Applying this to intersection classes and using homology relative 
to $-\infty$ to simplify integrals yields the following representation.
\begin{theorem} \label{th:residue}
Let $F = P/Q$ be the quotient of Laurent polynomials with Laurent
series $\sum_{\rr \in E} a_\rr \zz^{\rr}$ converging on $\TT (\xx)$
when $\xx \in B$, where $B$ is a component of $\amoeba(Q)^c$.  Fix 
$\rhat_*$ and assume the minimal stationary value $c_* (\rhat_*)$ is finite
and $K_\infty (\rhat_*)$ is empty.  
Let $B'$ denote a component of the complement of $\amoeba(Q)$ on which
$h_{\rhat_*}$ goes linearly to $-\infty$, as constructed prior to
Proposition~\ref{pr:vanish}.
Then for any $\xx \in B$ and $\yy \in B'$,
$$a_\rr = \frac{1}{(2 \pi i)^{d-1}} \int_{\I (\TT , \TT')}
   \Res \left (  \zz^{-\rr}F(\zz) \, \frac{d\zz}{\zz} \right ) \, $$
   for all but finitely many $\rr$.
If $P$ is replaced by any holomorphic function then the same 
representation of $a_\rr$ holds up to a function decreasing
super-exponentially in $|\rr|$.
\end{theorem}

\begin{proof}
If $P$ is polynomial then, for all but finitely many $\rr \in E$,
\begin{eqnarray*}
(2 \pi i)^{d-1} a_\rr & = & \frac{1}{2 \pi i} \int_{\TT (\xx)} 
   \zz^{-\rr} F(\zz) \, \frac{d\zz}{\zz} \\
& = & \frac{1}{2 \pi i} \int_{\leray \I (\TT , \TT')} \zz^{-\rr} F(\zz) 
   \, \frac{d\zz}{\zz} \; + \; 
  \frac{1}{2 \pi i}  \int_{\TT(\yy)} \zz^{-\rr} F(\zz) \, \frac{d\zz}{\zz} \\
& = & \frac{1}{2 \pi i} \int_{\leray \I (\TT , \TT')} \zz^{-\rr} F(\zz) 
   \, \frac{d\zz}{\zz} \, .
\end{eqnarray*}
The first line above is Cauchy's integral formula, the second is
Proposition~\ref{pr:int}, and the third is Corollary~\ref{cor:tube}
or Proposition~\ref{pr:vanish}.  By the Residue Theorem, 
$$ \frac{1}{2 \pi i} \int_{\leray \I (\TT , \TT')} \zz^{-\rr} F(\zz) =
   \int_{\I (\TT , \TT')} \zz^{-\rr} \Res \left ( F(\zz) 
   \, \frac{d\zz}{\zz} \right ) \, ,$$
proving the theorem when $P$ is a polynomial.  When $P$ is not polynomial,
use Theorem~\ref{th:low} in place of Proposition~\ref{pr:vanish}
in the last line.
\qed
\end{proof}

Combining Theorems~\ref{th:main2} and~\ref{th:residue} yields
the most useful form of the result: a representation of the
coefficients $a_\rr$ in terms of integrals over relative
homology generators produced by the stratified Morse decomposition.
In the following theorem we remove the assumption that $\sing_*$
is smooth, though we still use residues at the smooth points.

Let $\sigma_1, \ldots , \sigma_m$ enumerate the stationary points
of $\sing_*$ in weakly decreasing order of height $c_1 \geq c_2 \geq
\cdots \geq c_m$.  For each $j$, denote the relevant homology pair by
\begin{equation} \label{eq:GM}
(X_j^+ , X_j^+) := (\sing_{\leq c_j - \ee} \cup B_j ,
   \sing_{\leq c_j - \ee})
\end{equation}
where $B_j$ is a sufficiently small ball around $\sigma_j$ in $\sing$.
Let $k_j := \dim H_{d-1} (X_j^+ , X_j^-)$ and let $\beta_{j,1} , \ldots , 
\beta_{j,k_j}$ denote cycles in $H_{d-1} (X_j^+)$ that project to a 
basis for $H_{d-1} (X_j^+ , X_j^-)$ with integer coefficients.  

In the case where $\sigma_j$ is a smooth point of $\sing$, 
stratified Morse theory~\cite{GM} implies that $k_j = 1$ and 
$\beta_{j,1} = \leray \gamma_j$ is a cycle agreeing locally with 
a tube around the unstable manifold $\gamma_j$ for the downward 
$h_{\rhat}$ gradient flow on $\sing$.  This leads
to the following decomposition for $a_\rr$.

\begin{theorem}[stratified Morse homology decomposition] \label{th:SMT}
Let $F = P/Q$ be rational.  Fix $\rhat_*$, assume $K_\infty (\rhat_*)$
is empty, and enumerate the affine stationary points 
$\sigma_1 \, \ldots , \sigma_m$ as above. 
Then there are integers $\{ n_{j,i} : 1 \leq j \leq m, 
1 \leq i \leq k_j \}$ such that
\begin{equation} \label{eq:SMT residue}
a_\rr = \frac{1}{(2 \pi i)^{d-1}} \sum_{j=1}^m \sum_{i=1}^{k_j} 
   n_{j,i} \int_{\beta_{j,i}} 
   \zz^{-\rr} F(\zz) \frac{d\zz}{\zz} \, .
\end{equation}
When $\sigma_j$ is a smooth point of $\sing_*$, then $k_j = 1$,
the cycle $\beta_j$ agrees locally with a tube $\gamma_j$ around
the unstable manifold $\gamma_j$ at $\sigma_j$ for the downward 
gradient flow of $h_{\rhat_*}$ on $\sing$ and the corresponding summand
in~\eqref{eq:SMT residue} is given by
$$\int_{\gamma_j}
   \Res \left ( \zz^{-\rr} F(\zz) \frac{d\zz}{\zz} \right ) \, .$$
\end{theorem}

Theorem~\ref{th:SMT} is the culmination of this Section, and
provides a crucial (and highly desired) tool for analytic combinatorics 
in several variables.  Let $j_*$ be the least $j$ for which some 
$n_{j,i}$ is nonzero.  Generically the dominant asymptotic 
term is then the sum
of the terms in~\eqref{eq:SMT residue} with $n_{j,i} = 0$ and 
$h_{\rhat_*} (\sigma_j) = h_{\rhat_*} (\sigma_{j_*})$.  The
full expansion has several benefits over a statement of the
leading term only.  First, one might not know $j_*$.  In fact
in~\cite{BMP-lacuna}, the asymptotics of the diagonal coefficients
are settled only by computing all the residue integrals, then
determining the integers $\{ n_{j,i} \}$ via rigorous numerics. 
Secondly, knowing the subdominant terms allows one to compute
error estimates in the case where the exponential rates are very
close or are converging to one another.  Thirdly, sometimes
$G_j$ is generated by a local cycle, supported in an arbitrarily
small neighborhood of $\sigma_j$.  This is a natural choice for 
$\alpha_j$, whence the next asymptotic terms are meaningful.
Fourthly, one may want a trans-series expansion of $a_\rr$,
for which the contributions of all orders are needed.

\section{Computation of stationary points at infinity} 
\label{sec:computing}

We begin by recalling some background about stratifications and 
affine stationary points.

\subsubsection*{Computing a stratification}

To compute stationary points one requires a stratification.
Often, there is an obvious stratification; for example, polynomial
varieties are generically smooth, in which case the trivial 
stratification $\{ \sing \}$ suffices\footnote{Formally one must 
join with the stratification generated by the coordinate planes, 
so smooth manifolds intersecting coordinate subspaces nontransversely
(which again, is non-generic) might require refinement.} 
In non-generic cases, however, one must produce a stratification 
of $\sing$ before proceeding with the search for stationary points.

There are two relevant facts to producing a stratification.  One
is that there is a coarsest possible Whitney stratification, called
the {\em canonical Whitney stratification} of $\sing$.  It
is shown in~\cite[Proposition~VI.3.2]{Teissier1982} that there
are algebraic sets $\sing = F_0 \supset F_1 \supset \cdots
\supset F_m = \emptyset$ such that the set of all connected components
of $F_i \setminus F_{i+1}$ for all $i$ forms a Whitney stratification
of $\sing$ and such that every Whitney stratification of $\sing$ is
a refinement of this stratification.  This canonical stratification is
effectively computable: an algorithm exists, given $Q$, to
determine generators for the radical ideals 
corresponding to the Zariski closed sets $F_i$.

Mostowski and Rannou~\cite{mostowski-rannou} give an 
algorithm to compute stratifications using quantifier
elimination, leading to a bound on the computation time
which is doubly exponential in $m$; an alternative
algorithmic approach is presented 
in~\cite[Section~2]{dinh-jelonek2021}.
In our experience, the large doubly-exponential upper bound
is somewhat more pessimistic than what one has to 
deal with on actual combinatorial examples.

\subsubsection*{Computing the affine stationary points}

Assume now that a Whitney stratification $\{ \str_\alpha : \alpha \in A \}$
is given, meaning the index set $A$ is stored, along
with, for each $\alpha \in A$, a collection of polynomial generators
$f_{\alpha , 1} , \ldots , f_{\alpha , m_\alpha}$ for the radical 
ideal $\I (\str_\alpha)$.  The set $\str_\alpha$ is the algebraic
set  $V_\alpha := \V (f_{\alpha,1} , \ldots , f_{\alpha,m_\alpha})$ 
minus the union of varieties $V_\beta$ of higher codimension.
Potentially by replacing $\I (\str_\alpha)$
with its prime components, we may assume that the tangent space of 
$V_\alpha$ at any smooth point has constant codimension $k_\alpha$.  
After computing the canonical Whitney stratification, recall that
we refine if necessary to ensure that the defining ideal for each 
stratum of co-dimension $k$ has $k$ generators with linearly independent 
differentials at every point.  

By Definition~\ref{def:graph} the set of affine stationary 
points $\crit^{\rm aff} (\str_\alpha , \yy)$ in the direction $\yy$ 
is defined, after removing points in varieties of higher codimension, 
by the ideal 
containing the polynomials $f_{\alpha,1} , \ldots , f_{\alpha,m_\alpha}$
together with the $(c+1)\times(c+1)$ minors of $\J(\zz,\yy)$.
Taking the union over all strata $\str_\alpha$ produces all the 
affine stationary points.  This description simply restates the 
so-called `critical point equations' given 
in~\cite[(8.3.1-8.3.2)]{PW-book}, or, in the common special case of 
a stratum of co-dimension one, more explicitly by (8.3.3) therein. 

For the fixed integer vector $\rr$ and height interval $[a,b]$,
the inequalities $h_{\rhat} (\zz) \in [a,b]$ impose further semi-algebraic
constraints.  Unfortunately,  these increase the complexity 
considerably and behave badly under perturbations of the integer
vector $\rr$.  If one can compute open cones of values of $\rr$
in which the structure of the computation does not change, one can 
then pick a single $\rr$ in the cone to minimize complexity, 
making for a feasible computation.  Otherwise, one must settle 
for computations based on a fixed direction $\rr$.

\subsubsection*{Computing stationary points at infinity}

To determine whether there exist stationary points at infinity we use 
ideal quotients, corresponding to the difference of algebraic varieties. 
Recall that the variety $\V(I:J^{\infty})$ defined by the saturation 
$I : J^{\infty}$ of two ideals $I$ and $J$ is the Zariski closure of 
the set difference $\V(I) \setminus \V(J)$ (see~\cite[Section 4.4]{CLO1}), 
and can be determined through Gr{\"o}bner basis computations.  

\begin{definition}[saturated stationary point ideal $\fC_\alpha$]
\begin{itemize}
\itemsep1em
\item[$\bullet$]
For a stratum $\str_\alpha$ let $C_{\alpha}$ be the projective
ideal defining $\crit^{\rm aff} (\str_\alpha,\yy)$.  In other words, 
taking the homogenizing variable to be $z_0$, the ideal $C_\alpha$
is generated by the homogenizations in the $z$ variables of both
$f_{\alpha , 1} (\zz,\yy), \ldots , f_{\alpha , m_\alpha} (\zz,\yy)$ 
and the $(c+1) \times (c+1)$ minors of $\J (\zz , \yy)$.  
\item[$\bullet$]
Let $D_\alpha$ denote the ideal generated by $z_0 z_1 \cdots z_d$ and the 
homogenizations of all polynomials $f_{\beta , j}$ in strata of 
higher codimension $k_\beta > k_\alpha$.  
\item[$\bullet$]
Define $\fC_\alpha$ to be the result of saturating 
$C_\alpha$ by the ideal $D_\alpha$.  
\end{itemize}
\end{definition}
Geometrically, the variety $\V (\fC_\alpha)$ is the Zariski closure 
of $\V (C_{\alpha}) \setminus \left(\V(z_0 z_1 \cdots z_d) \cup 
\V(\fF_{j+1})\right)$, that is, the closure of that part of the 
graph of the relation $\zz \in \crit (\yy)$ in $\CP^d \times \CP^{d-1}$ 
corresponding to points $(\zz, \yy)$ whose $\zz$ component is not on 
a substratum and not at infinity (including the coordinate planes).  
Note that in this setting, the Zariski closure equals the classical 
topological closure~\cite[Theorem 2.33]{mumford}.  

\begin{definition}[saturated ideals] \label{def:sat}
Fix a stratum $\str_\alpha$.
\begin{itemize}
\itemsep1em
\item[$\bullet$]
Let $\fC_\alpha^\infty$ denote the result of substituting $z_0 = 0$ 
in $\fC_\alpha$.
\item[$\bullet$]
Let $\fC_\alpha^\infty (\rhat)$ denote the result of substituting 
$\yy = \rhat$ in $\fC_\alpha^\infty$.
\end{itemize}
\end{definition}

The variety $\V (\fC_\alpha^\infty)$ finds all SPAI.
The variety $\V (\fC_\alpha^\infty (\rhat))$ finds all SPAI 
in a given direction $\rhat$, that is, all limits of affine 
points in $\crit (\str_\alpha,\yy)$ with $\yy \to \rhat$.  
This is stated in the following proposition, whose proof 
follows directly from our definitions.

\begin{proposition}[computability of stationary points at infinity] 
\label{pr:infinity}
\begin{enumerate}[(i)]
\item
The rational function $F(\zz) = P(\zz)/Q(\zz)$ has SPAI
if and only if for some $\alpha$ there is a projective 
solution to $\fC_\alpha^{\infty}$, in other words, a
solution other than $(0, \ldots , 0)$.
\item
The rational function $F(\zz) = P(\zz)/Q(\zz)$ has SPAI
in direction $\rhat$ if and only if for some $\alpha$ there
is a projective solution to $\fC_\alpha^\infty (\rhat)$.
\end{enumerate}
\qed
\end{proposition}

Proposition~\ref{pr:infinity} computes a superset of what
we need: SPAI regardless of height.  We only care about those
at finite heights, indeed heights above the least affine stationary
value.  Unfortunately, we do not know a way to automate the
height computation, which is not polynomial; doing so is an
interesting problem for further research.

\begin{problem}
Find an effective way to compute $\crit_{[a,b]}(\rhat)$, for
irrational $\rhat$ or as a symbolic computation in $\rhat$.
\end{problem}

When $\rhat$ is rational, we can do a little better.  
In this case the height(s) may be computed from the start
along with the stationary points themselves because the
exponentiated heights are polynomial.  This gives the following
corollary, whose use is illustrated in the upcoming examples.

\begin{corollary} \label{cor:heights}
Let $\rr$ be an integer vector.  Introducing one more variable $\eta$,
let $H_\alpha$ denote the ideal generated by $C_\alpha$ along with 
$\eta z_0^{|\rr|} - \zz^\rr$.  Let $\fH_\alpha$ be the result of saturating
by $D_\alpha$, let $\fH_\alpha^\infty$ be the result of substituting
$z_0 = 0$, and let $\fH_\alpha^\infty (\rr)$ be the result of further
substituting $\yy = \rr$.  Then there exists a H-SPAI of height 
$\log c$ in direction $\rr$ if and only if there is a solution 
to $\fH_\alpha^\infty (\rr)$ with $\eta$-coordinate equal to $c$.
\end{corollary}

\section{Examples} \label{sec:examples}

When $Q$ is square-free and $\sing$ is smooth, $\sing$ itself forms 
a stratification.  The pseudocode in Algorithm~\ref{alg:1}
computes stationary points at infinity in this case
(the pseudocode in Algorithm~\ref{alg:2}
computes stationary points at infinity in the general case
but requires the algebraic sets defining the 
canonical Whitney stratification of $\sing$ as input).
We have implemented this algorithm, and more general variants, in Maple. 
A Maple worksheet with our code and examples is available from
ACSVproject.org (search for this paper) and the authors' webpages.

\begin{algorithm}
\DontPrintSemicolon
\KwIn{Polynomial $Q \in \mathbb{Z}[\zz]$ and direction $\rr \in \Z^d$ 
with $\sing(Q)$ smooth} \KwOut{Ideal $\fC'$ in the variables of $Q$ 
and a homogenizing variable $z_0$ such that there is a stationary point 
at infinity if and only if the generators of $\fC'$ have a non-zero solution.}
If $Q$ is not square-free replace it with its square-free part (the product of its distinct irreducible factors); \\
Let $\tilde{Q} = z_0^{\deg Q}Q(z_1/z_0,\dots,z_d/z_0)$; \\
Let $C$ be the ideal generated by $\tilde{Q}$ and
\[ y_jz_1 (\partial \tilde{Q}/\partial z_1) - y_1z_j (\partial \tilde{Q}/\partial z_j), \qquad (2 \leq j \leq d); \]
\\
Saturate $C$ by $z_0z_1\cdots z_d$ to obtain the ideal $\fC$; \\
Substitute $y_j=r_j$ for $1 \leq j \leq d$ and return the resulting 
ideal with the generator $z_0 z_1 \cdots z_d$ added. 
\caption{{\sf Stationary points at infinity (smoothness assumption)}}
\label{alg:1}
\end{algorithm}

\begin{algorithm}  
\DontPrintSemicolon
\KwIn{Polynomial $Q \in \mathbb{Z}[\zz]$, 
direction $\rr \in \Z^d$ and polynomial generators 
of algebraic sets $F_0 \supset F_1 \supset \cdots \supset F_m$ 
defining the canonical Whitney stratification of the zero set of $Q$}
\KwOut{Set of ideals $\mathcal{S}$ in the variables of $Q$ and 
a homogenizing variable $z_0$ such that there is a stationary point 
at infinity if and only if there exists $\fC' \in \mathcal{S}$ whose 
generators have a non-zero solution.}
Set $\mathcal{S} = \emptyset$ \\ 
For $j$ from 1 to $m-1$: \\
\qquad Compute the prime decomposition of ideals $F_j = P_1 \cap \dots \cap P_r$ \\
\qquad For each $I \in \{P_1,\dots, P_r\}$: \\
\quad\quad\quad Let $c$ be the codimension of $I$; \\
\quad\quad\quad Let $C$ be the ideal generated by $I$ together
with the $(c+1)\times(c+1)$ minors of the matrix \\
\quad\quad\quad $\J(\zz,\yy)$ in~\eqref{eq:Jmat} with the $f_i$ polynomials set to 
the generators of $I$; \\
\quad\quad\quad Homogenize $C$ in the $\zz$ variables with the 
new variable $z_0$ to obtain the ideal $C'$; \\
\quad\quad\quad Saturate $C'$ by $z_0z_1\cdots z_d$ to obtain 
the ideal $\fC$; \\
\quad\quad\quad Substitute $y_j=r_j$ into $\fC$ for $1 \leq j \leq d$ 
and put this ideal with 
$z_0 z_1 \cdots z_d$ added into $\mathcal{S}$; \\
Return $\mathcal{S}$
\caption{{\sf Stationary points at infinity (no smoothness assumption)}}
\label{alg:2}
\end{algorithm}

We give three examples in the smooth bivariate case, always computing
in the main diagonal direction $\rr=(1,1)$ and examining the 
coefficients $a_{n,n}$ of $1/Q$.  In the first example there is 
no affine stationary point; it follows that there must be a stationary point 
at infinity which determines the exponential growth rate.  
In the second there are both stationary points at infinity and 
affine stationary points, with the stationary point at infinity being too low 
to matter.  In the third, the stationary point at infinity is higher than 
all the affine ones and controls the exponential growth rate of diagonal 
coefficients.

\begin{example}[smooth case, stationary point at infinity] \label{eg:actual}
Let $Q(x,y) = 2 - xy^2 - 2xy - x + y$, so that $\sing$ is smooth and
we can use the above code for the diagonal direction 
$\rr = (1,1)$.  First, an examination of the polynomial system
$Q = \partial Q / \partial x - \partial Q /
\partial y = 0$ shows there are no affine stationary points.  
Thus, if there were no stationary points at infinity then
the diagonal coefficients of $Q(x,y)^{-1}$ would decay super-exponentially.  
It is easy to see that this does not happen, for example by extracting
the diagonal.  This may be done via the Hautus-Klarner-Furstenberg
method~\cite{hautus-klarner-diagonal,BostanDumontSalvy2017}, 
giving $\Delta Q(x,y)^{-1} = (1-z)^{-1/2}/2$.

We compute the existence of stationary points at infinity in the 
diagonal direction using our Maple implementation via the command
\begin{verbatim}SPatInfty(2 - x*y^2 - 2*x*y - x + y , [1,1])\end{verbatim}
This returns the ideal 
\begin{verbatim} [(H - 1)^2 , Z, y (H - 1), x ]\end{verbatim} 
where $Z$ is the homogenizing variable,  
meaning the projective point $(Z:x:y) = (0:0:1)$ is a
stationary point at infinity, which has height $\log |1| = 0$.
The stationary point at infinity is a topological 
obstruction to the gradient flow across height~0, pulling trajectories
to infinity; it suggests that the diagonal power series coefficients 
of $1/Q$ do not grow exponentially nor decay exponentially (in fact 
they decay like a constant times $1/\sqrt{n}$). Further geometric analysis 
of this example is found in~\cite[pages~120--121]{devries-thesis}.
\end{example}

\begin{example}[highest stationary point is affine] \label{eg:affine}
Let $Q(x,y) = 1 - x - y - xy^2$.  To look for stationary points at infinity
in the diagonal direction, we execute \verb+SPatInfty(Q,[1,1])+ to obtain 
the ideal
\begin{verbatim}
[(H + 1)(4H^2 + 4H - 1), Z, y(H + 1), x]
\end{verbatim} 
showing that $(x:y:Z) = (0:1:0)$ 
is a stationary point at infinity at height $\log|-1|=0$.
This time, there is an affine stationary point 
$(1/2 , \sqrt{2}-1)$ of greater height.  This affine point is easily
seen to be a topological obstruction and therefore controls 
the exponential growth.  Theorem~\ref{th:residue} allows us to 
write the resulting integral as a saddle point integral in $\sing_*$
over a class local to $(1/2, \sqrt{2}-1)$, thereby producing an
asymptotic expansion with leading term $a_{n,n} \sim c n^{-1/2} 
(2+\sqrt{2})^n$.  
\end{example}

\begin{example}[stationary point at infinity dominates affine points]
\label{eg:stat}
Let $Q(x,y) = -x^2 y-10 x y^2 - x^2 - 20 x y - 9 x + 10 y + 20$.
This time \verb+SPatInfty(Q,[1,1])+ produces 
\begin{verbatim} 
[(2 H^4 - 11 H^3 + 171 H^2 - 1382 H + 3220) (H-1)^2, Z, y (H-1), x]\end{verbatim} 
Again, because $x$ and $y$ cannot both vanish, the only height of a
stationary point at infinity is $\log |1| = 0$.
There are four of affine stationary points:  a Gr{\"o}bner basis
computation produces one conjugate
pair with $|xy| \approx 9.486$ and another conjugate pair with
$|xy| \approx 4.230$.  Both
of these lead to exponentially decreasing contributions, meaning
the point at infinity could give a topological obstruction for
establishing asymptotics. If there is an obstruction, it would
increase the exponential growth rate of diagonal
coefficients from $4.23^{-n}$ to no exponential growth or decrease.
To settle this, we can compute a linear differential equation satisfied
by the sequence of diagonal coefficients.
This reveals that the diagonal asymptotics are of order
$a_{n,n} \asymp n^{-1/2}$, meaning the exponential growth 
rate on the diagonal is in fact one.
\end{example}

\begin{example}[dominant asymptotics with no SPAI] 
\label{eg:spurious}
If $Q(x,y,z) = 1-x-y-z-xy$ then running \verb+SPatInfty(Q,[1,1,1])+
shows there are no SPAI. The two affine stationary points are computed easily,
\begin{eqnarray*}
\sigma_1 & = & \left ( - \frac{3 + \sqrt{17}}{4} , - \frac{3 + \sqrt{17}}{4} , 
   \frac{7 + \sqrt{17}}{8} \right ) \\[2ex]
\sigma_2 & = & \left ( - \frac{3 - \sqrt{17}}{4} , - \frac{3 - \sqrt{17}}{4} , 
   \frac{7 - \sqrt{17}}{8} \right ) \, ,
\end{eqnarray*}
so Theorem~\ref{th:SMT} implies
$$a_{n,n,n} = \frac{1}{(2 \pi i)^2} 
   \sum_{j=1}^2 \int_{\beta_j} \zz^{-\rr} 
   \Res \left ( F(\zz) \, \frac{d\zz}{\zz} \right )$$
where $\beta_1$ and $\beta_2$ are respectively the downward
gradient flow arcs on $\sing$ at $\sigma_1$ and $\sigma_2$.
Saddle point integration gives an asymptotic series for each, the 
series for $\sigma_1$ dominating the series for $\sigma_2$, yielding
an asymptotic expansion for $a_{n,n,n}$ beginning
$$a_{n,n,n} = \left(\frac{3+\sqrt{17}}{2}\right)^{2n}\left(\frac{7+\sqrt{17}}{4}\right)^{n} \cdot \frac{2}{n\pi \sqrt{26\sqrt{17}-102}}\left(1 + O\left(\frac{1}{n}\right)\right).$$
% Next order term is -6592/(17*sqrt(-102+26*sqrt(17))*(697+47*sqrt(17))*Pi)
\end{example}

The following application concerns an analysis in a case
where $\sing$ is not smooth.  There is an interesting singularity
at $(1/3,1/3,1/3,1/3)$ and a geometric analysis involving a
lacuna~\cite{BMP-lacuna}, which depends on there being no
stationary points at infinity at finite height.  This example
illustrates shows the full power of the results in 
Sections~\ref{ss:intersection}~--~\ref{ss:residues}.

\begin{example}[application to GRZ function] \label{eg:GRZ}
In~\cite{BMPS}, asymptotics are derived for the diagonal coefficients
of several classes of symmetric generating functions, including
the family of 4-variable functions $\{ 1 - x-y-z-w + C xyzw : C > 0\}$ 
attributed to Gillis, Reznick and Zeilberger~\cite{laguerre}.
The most interesting case is when the parameter $C$ passes through the
stationary value~27: diagonal extraction
and univariate analysis show the exponential growth rate of 
the main diagonal to have a discontinuous jump downward at the 
stationary value~\cite[Section~1.4]{BMPS}.  This follows from 
Corollary~\ref{cor:main} if we show there are in fact 
no stationary points at infinity above this height. Here there is 
a single point where the zero set of $Q$ is non-smooth,
the point $(1/3,1/3,1/3,1/3)$. Running Algorithm~\ref{alg:1} with
the modification to remove the single non-smooth point
yields the ideal \verb_[Z,z^4,-z+y,-z+x,-z+w]_
which has only the trivial solution $Z=w=x=y=z=0$.  Thus, there 
is no stationary point at infinity for the diagonal direction.

There are three affine stationary points, one at $(1/3,1/3,1/3,1/3)$,
one at $(\zeta, \zeta, \zeta, \zeta)$ and one at $(\overline{\zeta} , 
\overline{\zeta} , \overline{\zeta} , \overline{\zeta})$,
where $\zeta = (-1 - i \sqrt{2}) / 3$.  Call these points
$\xx^{(1)}$, $\xx^{(2)}$, and $\xx^{(3)}$.  The stationary point with the
greatest value of $\hr$ (in the main diagonal direction) is $(1/3,1/3,1/3,1/3)$,
which has height $\log 81$.  The other stationary points both have height $\log 9$.
By a somewhat involved topological process, it is checked 
in~\cite{BMP-lacuna} that $\pi \TT = 0$ in $H_4 (B 
\cup \M_{\log 81} - \ee , \M_{\log 81} - \ee)$, where $B$ is a 
small ball centered at $\xx^{(1)}$.  Crucially for the
analysis after that, it follows from Proposition~\ref{th:main2}
that $\TT$ is homologous to a chain supported in $\M_{c_2 + \ee}$. 
In other words, $\TT$ can be pushed down until hitting obstructions 
at $(\zeta, \zeta, \zeta, \zeta)$ and $(\overline{\zeta} ,  
\overline{\zeta} , \overline{\zeta} , \overline{\zeta})$.  In fact 
an alternative analysis using a differential equation satisfied by the 
diagonal verifies a growth rate of $9^n$, not $81^n$.
By means of Theorem~\ref{th:residue}, the coefficients $a_{n,n,n,n}$ 
may be represented as a residue integrals and put in standard saddle 
point form.  When this is done, one obtains the
more precise asymptotic $K n^{-3/2} 9^n$.  The value of $K$
depends on geometric invariants (curvature) and topological invariants 
(intersection numbers) and can be deduced by rigorous numeric methods.
Details are given in~\cite[Section~8]{BMP-lacuna}.
\end{example}

We end with a three-dimension example with SPAI that are irrelevant
for asymptotics.

\begin{example}[irrelevant SPAI in three dimensions] \label{eg:tri}
Consider the diagonal direction and
$$Q(x,y,z) = 1-x+y-z-2xy^2z. $$
There are two affine stationary points,
$\xx_\pm = \left(\frac{1}{3},\frac{9\pm\sqrt{105}}{4},\frac{1}{3}\right).$
Running our algorithm shows there is also a stationary point at infinity
of height $-\log|-1/2|=\log(2)$. Since the height of 
$\xx_-$ is $-\log\left|\frac{9-\sqrt{105}}{36}\right| > \log(2)$, 
the stationary point at infinity does not affect dominant asymptotics 
of $1/Q$.  One can get a mental picture of the situation by examining 
the Newton polytope of $Q$.  Due to the monomials $x,y,$ and $z$, 
the dual cone of the Newton polytope at the origin is the negative 
orthant.  Thus, the component $B$ of $\amoeba(Q)$ corresponding to 
the power series expansion of $1/Q$ admits the negative orthant 
as its recession cone.  This implies one cannot move along 
a direction perpendicular to $(1,1,1)$ and stay in $B$, 
so the stationary point at infinity
comes from other components of the amoeba complement.  
In fact, the stationary point at 
infinity lies on the closure of the complements of $\amoeba(Q)$ 
corresponding to the vertices $(0,1,0)$ and $(1,2,1)$ of the 
Newton polytope of $Q$; it can be directly verified that both of
these components have a recession cone containing a vector normal 
to the diagonal direction.  Ultimately, the lack of a stationary point 
at infinity of highest height implies an asymptotic expansion 
of the diagonal of $1/Q$ beginning
$$a_{n,n,n} = \left(-\frac{27+3\sqrt{105}}{2}\right)^{n} \cdot 
   \frac{\sqrt{3}}{2n\pi}\left(1 + O\left(\frac{1}{n}\right)\right).$$
\end{example}

\section{Proof of Theorem~{\protect\ref{th:main}}} \label{sec:proof}

In the log-space, the phase function $h$ becomes the linear function 
$\htt$ mapping $\xx$ to $\langle \rhat , \xx \rangle$. 
We denote by $d\htt_\str (\xx)$ the tangential differential, meaning 
the restriction of the differential of the phase function (in log space)
to the tangent space to $\str$ at the point $\xx \in \str$.  

\begin{lemma} \label{lem:1}
Assume that $\crit^\infty_{[a,b]} (\rhat) = \emptyset$ for some real
interval $[a,b]$.  Then for every neighborhood $\nbd$ of 
$\crit_{[a,b]}^{\rm aff} (\rhat)$, there is a $\delta > 0$ such that
\begin{equation} \label{eq:bd below}
\left | d\htt_\str (\xx) \right | \geq \delta 
\end{equation}
at every affine point $\xx \in \tsing \setminus \nbd$ whose height is in 
the interval $[a,b]$.  In particular, if there are no affine stationary
points then $|d\htt_\str (\xx)|$ is globally at least $\dd$ for some 
$\dd > 0$.
\end{lemma}

\begin{proof} 
It suffices to prove this separately for each of the finitely many strata.
We may therefore fix $\rhat$ and $\str$.  Also fixing $a < b$, we 
let $\str_{[a,b]}$ denote the intersection of a stratum $\str$
in the log space with the set of points having heights in $[a,b]$.
Assume towards a contradiction that the norm of the tangential differential 
is {\em not} bounded from below on $\str_{[a,b]} \setminus \nbd$ where
$\nbd$ is a neighborhood of $\crit^{\rm aff} (\rhat)$ in the log space.
Let $\xx_k$ be a sequence in $\str_{[a,b]} \setminus \nbd$ for which the 
left-hand side of~\eqref{eq:bd below}
goes to zero; this sequence has no limit points whose height lies 
outside of $[a,b]$ and no limit points in $\crit_{[a,b]}^{\rm aff} (\rhat)$.
There are also no affine limit points outside of $\crit (\rhat)$; this is
because if $\xx\to\yy$ with $\yy$ in a stratum $\str_y \subset \overline{\str}$
then $|d\htt_{\str_y} (\yy)| \leq \liminf_{\xx \to \yy} | d\htt_\str (\xx) |$,
since the projection of the differential onto a substratum is at most
the projection onto $\str$.

By compactness $\{ \xx_k \}$ must have a limit point $\xx \in \CP^d$.
It follows from ruling out noncritical points, and affine stationary points 
with heights inside or outside of $[a,b]$, that $\xx \in H_\infty$.
Passing to a subsequence if necessary, $\htt (\xx_k)$ converges to
a point $c \in [a,b]$.  The differential of the phase on the log space is the
constant co-vector $\rhat$.  Therefore, convergence of the norm of the 
tangential projection of the differential to zero is equivalent to the
projection $\yy_k$ of $\rhat$ onto the normal space to $\str$ at $\xx_k$
converging to $\rhat$ (here we make repeated use of the identification
of normal spaces to the phase function at different points of log-space).
The points $(\xx_k , \yy_k , \htt (\xx_k))$ have thus been shown to
converge to $(\xx , \rhat , c)$.  Furthermore, each $(\xx_k , \yy_k , 
\htt (\xx_k))$ is in $R(\str , \rhat)$ due to the choice of $\yy_k$
as a nonzero vector in the normal space to $\str$ at $\xx_k$.  The
sequence is therefore a witness to a stationary point at infinity
of height $c$, contradicting the hypothesis and proving the lemma.
\qed
\end{proof}

Recall from~\cite[Section~9]{mather12} the notion of a 
stratified vector field.  Given a Whitney stratification of $\logspace$,
a {\bf stratified vector field} $\vv$ is defined to be a collection
$\{ \vv_\str \}$ of smooth sections of $(\str , T\str)$.
Greatly summarizing Sections~7 and~8 of~\cite{mather12}, 
a stratified vector field is said to be {\bf controlled} if 
for any strata $\str \subseteq \partial \str'$,
\begin{enumerate}[(i)]
\item
$\langle d\rho , \vv_{\str'} \rangle = 0$ where $\rho$
is the squared distance function to $\str$ in a given local
product structure, and 
\item
projection from $\str'$ to $\str$ in this local product structure 
maps $\vv_{\str'}$ to $\vv_\str$.  
\end{enumerate}

Proposition~9.1 of~\cite{mather12}, with $P = \R$, $f = h$ and
$\zeta$ the constant vector field $- d/dx$ on $\R$, says that 
there is a controlled lift of $\zeta$ to $\logspace$, that is
a controlled stratified vector field $\vv$ mapping by $h_*$ to $\zeta$.
This will be almost enough to prove Theorem~\ref{th:main}.  What we
need in addition is a uniform bound on $|\vv|$.  Although Mather
(and Thom before him) was not interested in bounding
$|\vv|$ (indeed, their setting did not allow a meaningful metric),
his proof in fact gives such a bound, as we now show.

\begin{lemma} \label{lem:2}
Let $a < b$ be real and suppose $\{ \str_\alpha \}$ is a 
Whitney stratification of $\tsing$ for which there are no 
finite or infinite stationary points with heights in $[a,b]$.
Assume a given set of projections and distance functions satisfying
the control conditions of~\cite[Sections~7-8]{mather12}.
Then there is a vector field $\vv$ in the log space, in other words
a section of $T \logspace$, over the base set 
$\logspace \cap \htt^{-1} [a,b]$, with the following properties.
\begin{enumerate}[(i)]
\item control: $\vv$ is a controlled stratified vector field
for the stratum $\str$.
\item unit speed: $\langle d\htt , \vv \rangle \equiv -1$.
\item regularity: $\vv$ is bounded. 
\end{enumerate}
\end{lemma}

Before beginning the proof we motivate with a shorter argument 	
that does not quite hold water.  By Lemma~\ref{lem:1},
the negative unit gradient vector field $\vv_\str$ on each stratum
has magnitude at least $\delta$.  For each point $x$ in a stratum $\str$, 
take a neighborhood $\nbd_x$ in $\C_*^d$ intersecting no lower dimensional
strata and for which $\vv_\str$ extends continuously to a vector field
$\vv_x$ for which $\langle d\htt , \vv_x \rangle \geq \delta / 2$.
Piece these together with a partition of unity.  By convexity the
resulting vector field $\vv$ has norm at most~1 everywhere.  By linearity
$\langle d\htt , \vv_x \rangle \geq \delta / 2$ everywhere.
This is the natural argument.  The gap is that the local product 
structure does not, on the surface, guarantee a bounded continuous
extension of $\vv_\str$.  This must be argued; however as mentioned
above, it follows from Whitney's conditions and is implicit in Mather's
arguments.

\begin{proof}
By Lemma~\ref{lem:1}, the tangential differential $d\htt_\Sigma$ 
is globally bounded from below by some positive quantity $\delta$.
Hence, the hypotheses of~\cite[Proposition~9.1]{mather12} are satisfied
with $V = \logspace$ and the given Whitney stratification, $P = \R$,  
$f = h$ and $\zeta$ the negative unit vector field on $\R$.
Following the proof of~\cite[Proposition~9.1]{mather12}, which 
already yields conclusions~$(i)$ and~$(ii)$ of the lemma, the only place
further argumentation is needed for boundedness and continuity is at the 
bottom of page~494.  There, $\vv_\str$ is constructed inductively given
$\vv_{\str'}$ on all strata of lower dimension.  The word ``clearly''
in the third line from the bottom of page~494 hides some linear
algebra which we now make explicit.

We assume that $X$ and $Y$ are strata, with $X \subseteq \partial Y$
and $\dim X = m < \ell = \dim Y$.
First we straighten $X$ near a point $x \in X$.  In a neighborhood $\nbd$
of $x$ in the ambient space $\R^{2d}$ there is a smooth coordinatization
such that $\R^m \times \bzero$ maps to $X$.  Applying another linear
change of coordinates if necessary, we can assume that $\vv_X (x) = e_1$.
This requires a distortion of magnitudes of tangent vectors by
$1 / |d\htt_X|$ at the point $x$; taking $\nbd$ small enough, we can
assume that the distortion on the tangent bundle over $\nbd$ is 
bounded by twice this, hence globally by at most $2 / \dd$.

Whitney's Condition A stipulates that as $y \in Y$ converges to 
a point $x \in X$, any limit $\ell$-plane of a sequence $T_{y} (Y)$ 
must contain $T_x X = \R^m \times {\mathbf 0}$.  This implies for 
$1 \leq j \leq m$ that the distance of $e_j$ to $T_{y} (Y)$ goes to 
zero (recall we have identified the tangent spaces $T_y (Y)$ for
different $y$).  Hence, for $1 \leq k \leq m$ there are vectors 
$c_k (y) \in (\R^m)^\perp$ going to zero as $y \to x$,
such that $e_k + c_k (y) \in T_y (Y)$.  These vectors $c_k (y)$ 
may be chosen as smooth functions of $y \in Y$, continuously as 
$y \to x \in X$; this follows because the tangent planes $T_y (Y)$ 
vary continuously with $y \in Y$ and semicontinuously as $y \to x \in X$, 
meaning that $T_x (X)$ is contained in the liminf of $T_{y_n} (Y)$.
The vectors $f_k := e_k + c_k (y)$ are linearly 
independent for $1 \leq k \leq m$ because their projections to the
first $m$ coordinates are linearly independent; hence they may 
be completed to a basis $\{ f_1 , \ldots , f_\ell \}$ of $T_y (Y)$.

Let $\{ y_n \}$ be a sequence of points in $Y$ converging to $x \in X$.  
Write $y_n$ as $(x_n,z_n)$ with $x \in \R^m$ in local coordinates and 
$z \in \R^{2d - m}$.  Apply Whitney's Condition B applied to the 
sequences $\{ y_n \}$ in $Y$ and $\{ x_n \}$ in $X$.  Passing to a
subsequence in which $T_{y_n} (Y) \to \tau$ and $z_n / |z_n| \to u$,
Whitney's Condition B asserts that $u \in T_x (Y)$.  Because $T_x (Y)
\subseteq \tau$, it follows that the distance between 
$z_n / |z_n|$ and $T_{y_n} (Y)$ goes to zero.  Hence there is a
sequence $c_{m+1} (y_n) \to \bzero$ in $(\R^m)^\perp$ as $n \to \infty$
such that $z_n / |z_n| + c_{m+1} (y_n) \in T_{y_n} (Y)$ for all $n$.
Because $c_{m+1} \in (\R^m)^\perp$, it follows that we may choose
the basis $\{ 1_j : 1 \leq j \leq \ell \}$ so that $f_{m+1} (y)
= z(y) / |z(y)| + c_{m+1} (y)$.

Now we have what we need to construct a controlled vector field on 
$Y$ that is controlled and close to $e_1$.  In fact we can construct
one that is spanned by $f_1 , \ldots , f_{m+1}$ using
linear algebra.  Write
\begin{equation} \label{eq:lin alg}
\vv (y) = \sum_{j=1}^{m+1} a_j (x(y),z(y)) f_j (y) \, .
\end{equation}

Guessing at the solution, we impose the conditions 
$\langle \vv , e_j \rangle = \delta_{1,j}$ for $1 \leq j \leq m$.
This implies the second control condition, namely that $\vv$ is
preserved by projection from $Y$ to $X$.  The first control condition,
preservation of distance along the vector field, means that
$\langle \vv , z/|z| \rangle = 0$ in local coordinates.  
Thus we arrive at the following system.
\begin{eqnarray*}
\langle \vv , e_1 \rangle = \sum_{j=1}^{m+1} a_j \langle f_j , e_1 \rangle = 
   a_1 + a_{m+1} \langle c_{m+1} , e_1 \rangle & = & 1 \\
\langle \vv , e_2 \rangle = \sum_{j=1}^{m+1} a_j \langle f_j , e_2 \rangle = 
   a_2 + a_{m+1} \langle c_{m+1} , e_2 \rangle & = & 0 \\
& \vdots & \\
\langle \vv , e_m \rangle = \sum_{j=1}^{m+1} a_j \langle f_j , e_m \rangle = 
   a_m + a_{m+1} \langle c_{m+1} , e_m \rangle & = & 0 \\
\langle \vv , \frac{z}{|z|} \rangle 
   = \sum_{j=1}^{m+1} a_j \langle f_j , \frac{z}{|z|} \rangle 
   = a_{m+1} + \sum_{j=1}^m a_j \langle c_j , \frac{z}{|z|} \rangle & = & 0
\end{eqnarray*}
We may write this as $(I + M) {\bf a} = (1,0,\ldots ,0)^T$ where $I$ is
the $(m+1) \times (m+1)$ identity matrix and $M \to 0$ smoothly as 
$y \to x$.  Therefore, in a neighborhood of $x$ in $Y$, the solution 
exists uniquely and smoothly and converges to $e_1$ as $y \to X$.
This allows us to extend $\vv$ from $X$ to a controlled vector field 
in a neighborhood of $X$ in $Y$, varying smoothly on $Y \setminus X$
and extending continuously to $X$, with the properties that
$\langle d\htt_Y , \vv \rangle \equiv -1$ and that the norm on $Y$ is 
at most a bounded multiple of the norm on $X$.  

The solution of the control conditions together with the condition 
$\langle d\htt_Y , \vv \rangle \equiv -1$ are a convex set; a 
partition of unity argument as in~\cite{mather12} preserves 
continuity, and, by the triangle inequality, global boundedness 
of~$\vv$.
\qed
\end{proof}

We are now ready to complete the proof of Theorem~\ref{th:main}.

\begin{proof}[of conclusion~$(i)$ of Theorem~\ref{th:main}]
Choose a vector field $\vv$ as in the conclusion of Lemma~\ref{lem:2}.
Let $\D$ be the set $\logspace \cap \htt^{-1} [a,b]$, inheriting
stratification from the pair $(\logspace , \tsing)$.  On each
stratum of $\D$ the vector field $\vv$ is smooth and bounded.
Let $\D'$ be the space-time domain $\{ (x,t) \in \D \times \R^+ : 
t \leq \htt (x) - a \}$.  Let $\flow : \D' \to \D$ be a solution 
to the differential equation 
$$\frac{d}{dt} \flow (x,t) = \vv \left ( \Psi (x,t) \right ) \, ; 
   \qquad \flow (x,0) = x \, .$$
Such a flow exists on each stratum because because $\vv$ is smooth 
and bounded, and the phase exit the interval $[a,b]$ in finite time 
$\tilde{h}(x) - a \leq b-a$; therefore each trajectory stays within 
a uniformly bounded vicinity of its starting point.  Hence, by 
compactifying the support of the vector field if necessary, the flow 
exists for the time necessary for the height to drop below $a$ from 
any point of the stratum, and is smooth, hence is well-defined everywhere 
on the stratum.
As in~\cite[Section~10]{mather12}, the local
one-parameter group of time-$t$ maps are all injective.  Because the
flow preserves the squared distance functions to strata, the flow
started on a stratum $X$ cannot reach a stratum $Y \subseteq \partial X$;
hence the time-$t$ maps preserve strata.
Continuity of the vector field implies continuity of the time-$t$ maps.

Extend the flow to $\D \times [0,b-a]$ by setting $\flow (x,t) = 
\flow(x,\htt(x)-a)$ for $t > \htt (x) - a$.  By construction the
flow is tangent to strata, hence the flow is stratum preserving.
The flow is continuous because the velocity is bounded and the
stopping time $\htt (x) - a$ is continuous.  For any stratum $\str$,
the flow defines a homotopy within $\str_{\leq b}$ whose final cross 
section is in $\str_{\leq a}$.
\qed  
\end{proof}

\begin{proof}[of conclusion~$(ii)$ of Theorem~\ref{th:main}]
Still in the logspace, 
for any $r \geq 0$, let $\nbd_r$ denote the union of closed $r$-balls 
about the affine stationary points; in particular $\nbd_0$ is the set
of stationary points.  
Due to the presence of affine stationary points, we may no longer invoke 
the conclusions of Lemma~\ref{lem:2}.  We claim, however, that the
conclusions follow if $(iii)$ is replace by $(iii')$: $\vv$ is
bounded and continuous outside any neighborhood $\nbd_r$ of the 
affine stationary points.  This follows from the original proof because 
the distortions are bounded by constant multiples of the quantities 
$1 / |d\htt_X|$, while by Lemma~\ref{lem:1} $|d\htt_X|$ is uniformly 
bounded away from zero outside any neighborhood of the affine stationary 
points.  
Fix the vector field $\vv = \vv_r$ satisfying conclusions 
$(i)$ and $(ii)$ of Lemma~\ref{lem:2} and $(iii')$ above 
and let $K = K(r)$ be an upper bound for $|\vv (r)|$
outside of $\nbd_{r/2}$.  

We build a deformation in two steps.  Fix $r > 0$ and set 
$\ee = r / K(r)$.  Next, apply Lemma~\ref{lem:2} with 
$[c + \ee , b]$ in place of $[a,b]$, which we can do because there 
are no stationary points with heights in $[c+\ee , b]$.  This produces a 
deformation retract of $h^{-1} (-\infty , b]$ to $h^{-1} (-\infty , c + \ee]$.  

Now we compose this with a slowed down version of $\vv$.
The hypotheses of no affine or infinite critical values in $[a,b]$
other than at $c$ imply the same hypotheses hold over a slightly
large interval $[a-\theta , b ]$.  Let $\chi : [0,\infty) \to [0,1]$
be a smooth nondecreasing function equal to~1 on $[1,\infty)$ and~0 
on $[0,1/2)$.  Define the vector field $\ww = \ww_r$ on 
$h^{-1}[a-\theta , b]$ by 
$$\ww (x) := \chi \left ( ||x^{\phantom{.}} , \nbd_0 ||_{\phantom{.}} 
   \right )  \, \cdot \,  
   \chi \left ( \frac{ h(x) - (a-\theta)}{\theta} \right ) \, 
   \, \cdot \, \vv (x) \, ,$$ 
where $||\cdot , \cdot||$
denotes distance.  Because $\vv$ defines a flow, so does $\ww$, 
the trajectories of which are precisely the trajectories of $\vv$ 
slowed down inside $\nbd_r$ and below height $a$.  
Note that the trajectories of $\ww$ stop completely 
inside $\nbd_{r/2}$ and below height $a - \theta$, however, trajectories 
not in these regions slow down so as never to reach $\nbd_{r/2}$ 
nor height $a - \theta$.

Run the flow defined by $\vv$ at time $2 \ee$.  Let $\str$ be any stratum 
and let $x$ be any point in $\str_{\leq c+\ee}$.  
Let $\tau_x : [0,\ee] \to \str$ be the trajectory 
started from $x$ (the trajectory remains in $\str$ because $\vv$ is a
controlled stratified vector field).  We claim that the time-$2\ee$ 
map takes $x$ to a point in $\str_{< c} \cup \nbd_{3r}$.  Because $r$
is arbitrary, this is enough to prove theorem.  The flow 
decreases height, so we may assume without loss of generality that 
$x \in h^{-1} [c,c+\ee]$.  If the trajectory never enters $\nbd_r$
then height decreases at speed~1, hence $\tau_x (2\ee) \in \str_{\leq c-\ee}$
and the claim follows.  If the $\tau_x (2\ee) \in \nbd_r$ then
trivially the claim is true.  Lastly, suppose the trajectory enters
$\nbd_r$ and leaves again.  Let $s \in (0,2\ee)$ be the last time
that $\tau_x (s) \in \nbd_r$.  Because $|\vv| \leq K$, we have
$||\tau(2\ee) , \nbd_0|| \leq r + 2 \ee K \leq 3r$, that is,
$\tau (2\ee) \in \nbd_{3r}$, finishing the proof of the claim.
\qed
\end{proof}

\section{acknowledgements}
The authors would like to thank Paul G{\"o}rlach for his advice 
on computational methods for determining stationary points at infinity,
to Justin Hilburn for ideas on the
proof of the compactification result and to Roberta Guadagni for
related conversations. The authors also thank anonymous referees
for pointing out related results from the literature and helping
to clarify and simplify our arguments.

\appendix
\section{Appendix} \label{sec:A1}

We now give an abstract answer to~\cite[Conjecture~2.11]{pemantle-mvGF-AMS},
in a manner suggested to us by Justin Hilburn and Roberta Guadagni.
Let $H(\zz) := \zz^\mm$ be the monomial function on $C_*^d$, 
and $G\subset C_*^d \times \C^*$, its graph. An easy case of 
toric resolution of singularities (see, e.g., 
\cite{khovanskii_newton_1978}) implies the following result.

\begin{theorem}
There exists a compact toric manifold $K$ such that 
$C_*^d$ embeds into it as an open dense stratum and the 
function $H$ extends from this stratum to a smooth 
$\mathbb{P}^1$-valued function on $K$.
\end{theorem}

\begin{proof}
The graph $G$ is the zero set of the polynomial $P_\mm:=h-\zz^\mm$ on $C_*^d\times\C^*$, where $h$ is the coordinate on the second factor. Theorem 2 in~\cite{khovanskii_newton_1978} implies that a compactification of $C_*^d\times\C^*$ in which the closure of $G$ is smooth exists if the restrictions of the polynomial $P_\mm$ to any facet of the Newton polyhedron of $P_\mm$ is nondegenerate (defines a nonsingular manifold in the corresponding subtorus). In our case, the Newton polytope is a segment, connecting the points $(\mm,0)$ and $(\zero,1)$, and this condition follows immediately. Hence, the closure of $G$ in the compactification of $C_*^d\times\C$ is a compact manifold $K$. 
We notice that the projection to $C_*^d$ is an isomorphism on $G$, and therefore $K$ compactifies $C_*^d$ in such a way that $H$ lifts to a smooth function on $K$.

Lifting the variety $\sing_*$ to $G\subset K$ and taking the closure produces the desired result: a compactification of $\sing_*$ in a compact manifold $K$ on which $H$ is smooth. 
\qed
\end{proof}

A practical realization of the embedding requires construction of 
a simple fan (partition of $\Real^{d+1}$ into simplicial cones with 
unimodular generators) which subdivides the fan dual to the Newton polytope 
of $h-\zz^\mm$.  While this is algorithmically doable (and implementations 
exist, for example in {\tt macaulay2}), the resulting fans depend strongly 
on $\mm$, and the resulting compactifications $K$ are hard to work with.

\begin{definition}[{\bf compactified stationary point}]
Define a compactified stationary point of $H$, with respect to 
a compactification of $C_*^d$ to which $H$ extends smoothly, 
as a point $\xx$ in the closure of $\sing$ such that $dH$ vanishes 
at $\xx$ on the stratum $\Sb(\xx)$, and $H(\xx)$ is not zero or infinite.
\end{definition}

Applying basic results of stratified Morse theory~\cite{GM} to $\compact$
directly yields the following consequence.

\begin{corollary}[no compactified stationary point implies Morse results]
\label{cor:abstract}
~~\\[2ex] $(i)$ If there are no stationary points or compactified stationary points
with heights in $[a,b]$, then $\sing_{\leq b}$ is homotopy equivalent 
to $\sing_{\leq a}$ via the downward gradient flow.

$(ii)$ If there is a single stationary point $x$ with critical value
in $[a,b]$, and there is no compactified stationary point with height
in $[a,b]$, then the homotopy type of the pair $\left(\M_{\leq b} \, , \, 
\M_{\leq b} \right)$ is determined by a neighborhood of $x$, with an
explicit description following from results in~\cite{GM}.
\end{corollary}

\bibliographystyle{alpha} 
\bibliography{bibl}

\end{document}